\documentclass[11pt]{article}
\usepackage[utf8]{inputenc}
\usepackage[english]{babel}
\usepackage{ulem}
\usepackage{graphicx}
\usepackage{amsmath,esint}
\usepackage{amssymb}
\usepackage{amsfonts}
\usepackage{amsthm}
\usepackage[left=2.1cm,top=1.5cm,right=2.1cm, bottom=2.2cm,letterpaper]{geometry}
\usepackage{mathrsfs}
\usepackage{caption}
\usepackage{subcaption}
\usepackage{latexsym}
\usepackage{tcolorbox}
\usepackage{mleftright}
\usepackage{enumitem}
\usepackage{float}
\usepackage{hyperref}
\usepackage[all]{hypcap}
\usepackage{autonum}
\usepackage{url}
\usepackage[toc]{appendix}
\newtheorem{theorem}{Theorem}[section]

\newtheorem{lemma}[theorem]{Lemma}
\newtheorem{definition}[theorem]{Definition}
\newtheorem{proposition}[theorem]{Proposition}

\newtheorem{remark}[theorem]{Remark}
\numberwithin{equation}{section}
\numberwithin{figure}{section}

\newcommand{\N}{\mathbb{N}}
\newcommand{\R}{\mathbb{R}}

\newcommand{\esssup}{\mathop{\mathrm{ess\,sup}}}
\newcommand{\essinf}{\mathop{\mathrm{ess\,inf}}}
\makeatletter
\renewcommand*\env@matrix[1][*\c@MaxMatrixCols c]{%
	\hskip -\arraycolsep
	\let\@ifnextchar\new@ifnextchar
	\array{#1}}
\makeatother
\def\const{{\mathrm{const}}}
\def\neweq#1{\begin{equation}\label{#1}}
	\def\endeq{\end{equation}}

\def\div{\hbox{\rm div}\,}

\def\loc{{\mathrm{loc}}}
\def\R{{\mathbb R}}
\def\N{{\mathbb N}}
\def\Fe{{\mathcal F}}
\def\Ha{{\mathscr H}}

\def\ve{{\mathbf v}}
\def\ue{{\mathbf u}}

\def\UU{{\mathbf U}}

\def\hp{\widehat{p}}
\def\cc{{\tilde c}}

\def\tu{{\tilde{\mathbf u}}}
\def\fe{{\mathbf f}}
\def\vp{{\varphi}}

\def\bbe{{\boldsymbol{\eta}}}

\def\n{{\mathbf n}}
\def\0{{\mathbf 0}}

\def\WO{{\widetilde\Omega}}

\def\WD{{\widetilde D}}

\def\GG{{\Gamma}}
\def\OO{{\Omega}}
\def\aa{{\mathbf a}}
\def\be{{\mathbf b}}
\def\ee{{\mathbf e}}
\def\we{{\mathbf w}}
\def\ae{{\mathbf a}}
\def\e{{\varepsilon}}

\def\fe{{\mathbf f}}

\date{}
\begin{document}

\title{The steady Navier-Stokes equations in a system\\  of unbounded channels with sources and sinks}
\author{Filippo Gazzola -- Mikhail Korobkov -- Xiao Ren -- Gianmarco Sperone }
\maketitle

\begin{abstract}
\noindent	
The steady motion of a viscous incompressible fluid in a junction of unbounded channels with sources and sinks is modeled through the Navier-Stokes equations under inhomogeneous Dirichlet boundary conditions. In contrast to many previous works, the domain is not assumed to be simply-connected and the fluxes are not assumed to be small. In this very general setting, we prove the existence of a solution with a
uniformly bounded Dirichlet integral in every compact subset. This is a~generalization of the classical Ladyzhenskaya--Solonnikov result obtained under the additional assumption of zero boundary conditions.
 For small data of the problem we also prove the unique solvability and attainability of Couette--Poiseuille flows at infinity. The main novelty of our approach is the proof of the corresponding  Leray--Hopf-type inequality by 
Leray's  \textit{reductio ad absurdum} argument
(since the standard Hopf cut-off extension procedure does not work for general boundary data). For this contradiction approach, we use some fine properties of weak solutions to the Euler system based on Morse--Sard-type theorems in Sobolev spaces obtained by
Bourgain, Korobkov \& Kristensen.
\end{abstract}

\tableofcontents

\section{Introduction}

In the plane $\mathbb{R}^2$ consider the strip $\Omega \doteq \R\times(0,1)$. Suppose that a viscous incompressible fluid (having constant kinematic viscosity equal to $1$) is allowed to flow in stationary regime in $\Omega$, in the absence of external forces. The velocity vector field $\mathbf{u} : \Omega \longrightarrow \mathbb{R}^2$ and scalar pressure $p : \Omega \longrightarrow \mathbb{R}$ characterizing the motion of the fluid are governed by the steady-state Navier-Stokes equations in $\Omega$ under no-slip boundary conditions on $\partial \Omega$:
\begin{equation}\label{poiseuille2d}
	\left\{
	\begin{aligned}
		&- \Delta \mathbf{u} + (\mathbf{u} \cdot \nabla) \mathbf{u} + \nabla p = \0 \, , \quad \nabla \cdot \mathbf{u} = 0 \ \ \mbox{ in } \ \ \Omega \, , \\[4pt]
		&\mathbf{u} = \0 \ \ \mbox{ on } \ \ \partial \Omega \, .
	\end{aligned}
	\right.
\end{equation}
It is well-known that, for any $\Fe \in \mathbb{R}$, an explicit solution to \eqref{poiseuille2d} is given by the \textit{Poiseuille flow} having \textit{transverse flux} $\Fe$, meaning that the pair
\begin{equation} \label{eq:Poi}
	\ue(x,y)= ( 6\Fe(y - y^2), 0) \qquad \text{and} \qquad p(x,y)=-12\Fe x \qquad \forall (x,y) \in \overline{\Omega} \, ,
\end{equation}
forms a classical solution to \eqref{poiseuille2d}, see \cite[Chapter II]{landau}. Moreover, since $\ue$ is divergence-free \eqref{poiseuille2d}$_1$, the transverse flux is constant through any cross section of the strip $\Omega$, that is,
$$
\int_0^1 (\ue(x,y) \cdot \ee_1 ) \, dy = 6\Fe \int_0^1 (y - y^2) \, dy = \Fe \qquad \forall x\in \mathbb{R} \, .
$$

Suppose now that the strip $\Omega$ is replaced by a smooth, unbounded and \textit{distorted} channel like the one depicted in Figure \ref{fig:lerayoriginal} below, having two \textit{rectangular outlets} $\Omega_1$ and $\Omega_2$ outside a bounded set $\Omega_{0}$. This means that, in terms of some rigid transformation $z \in \mathbb{R}^2 \mapsto (x_j(z), y_j(z))  \in \mathbb{R}^2$, they can be written as
\begin{equation}  \label{ch0}
	\Omega_j= \{ z \in\R^2 \ | \ x_j(z)>0 \, , \quad 0 < y_j(z) < h_j \ \} \quad \text{for some} \ h_{j} > 0 \, , \ \forall j \in \{1,2 \} \, .
\end{equation}
\vspace{-7mm}
\begin{figure}[H]\begin{center}
		\includegraphics[scale=0.7]{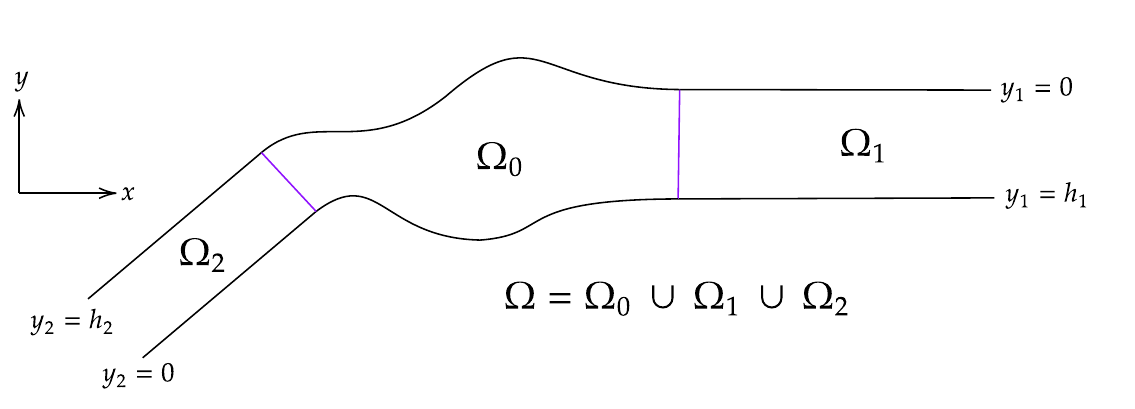}
	\end{center}
	\vspace{-5mm}
	\caption{A smooth, unbounded and distorted channel.}
	\label{fig:lerayoriginal}
\end{figure}
\noindent
Let $\Omega \subset \mathbb{R}^2$ be the domain depicted in Figure \ref{fig:lerayoriginal}, corresponding to an \textit{admissible domain} in the sense
of Amick \cite[Definition 1.1]{amick1977steady}. The pair $(\ue,p)$ in \eqref{eq:Poi} is no longer a solution to \eqref{poiseuille2d} in the distorted
channel $\Omega$. Nevertheless, given $\Fe \in \mathbb{R}$, each outlet $\Omega_1$ and $\Omega_2$ has an associated Poiseuille flow,
with flux $\mathcal{F}$, given explicitly in local coordinates by
\begin{equation}\label{Pflow}
\mathbf{P}_j(x_{j}, y_{j}) \doteq \left( \dfrac{6 \mathcal{F}}{h^{3}_{j}} y_{j} (h_{j} - y_{j}) , 0 \right) \quad \text{and} \quad
q_{j}(x_{j},y_{j})=- \dfrac{12\Fe}{h^{3}_{j}} x_{j} \qquad \forall (x_{j}, y_{j}) \in \overline{\Omega_{j}} \, , \ \forall j \in \{1,2 \} \, .
\end{equation}
During his visit to Leningrad in 1958 \cite[Remark 1.6]{galdi2008hemodynamical}, Jean Leray proposed \footnote{Quite interestingly, there seems to be no article written by Jean Leray containing a description of this problem that, nowadays, bears his name. One of the oldest mentions to Leray's problem can be found in the survey notes by Robert Finn from 1965, see \cite[Section 7]{finn1965}.} to Olga Ladyzhenskaya the problem of finding a classical solution $(\mathbf{u}, p) \in \mathcal{C}^{2}(\overline{\Omega}) \times \mathcal{C}^{1}(\overline{\Omega})$ to
\eqref{poiseuille2d} in this domain $\Omega$, with the following additional condition at infinitely large distances in the outlets:
\begin{equation} \label{poinf}
\ue \longrightarrow \mathbf{P}_j \quad \text{as} \quad |z| \to \infty \quad \text{in} \quad \Omega_{j} \, , \ \ \forall j \in \{1,2 \} \, ,
\end{equation}
that is, the velocity field should approach, in the far field, the corresponding Poiseuille flow in every outlet. After the unsuccessful
attempt by Ladyzhenskaya \cite{ladyzhenskaya1959stationary} in 1959, the problem was tackled many years later by Fraenkel
\cite{fraenkel1973theory,fraenkel1975theory} using perturbation methods to analyse a symmetric channel with slowly curving walls.
A breakthrough was obtained by Amick \cite[Theorem 3.8]{amick1977steady} in 1977, who provided an affirmative answer to \textit{Leray's problem}
for sufficiently small values of $\Fe$; his results also cover the cases of domains $\Omega$ {\it containing obstacles} (hence, non-simply
connected) and with {\it multiple rectangular outlets} (system of unbounded channels), see \cite[Remarks (b)-(c), p.498]{amick1977steady}.
All the boundary is covered by {\it homogeneous} (no-slip) Dirichlet boundary conditions and the equations contain {\it no external force}.
Amick does not address the uniqueness of the solution, nor the existence of a solution satisfying \eqref{poinf} for large values of $\Fe$,
problems that are still open up to date. In the smallness regime, we highlight results by Galdi \cite[Section XIII.2]{galdi2011introduction}
regarding unique solvability of Leray's problem, and those of Amick \cite{amick1978properties} and Horgan \& Wheeler \cite{horgan1978spatial}
ensuring exponential pointwise convergence to the Poiseuille flows in the far field.\par
In spite of these important contributions, a full answer to the original
Leray problem remains as one of the most challenging, yet unsolved, problems in Mathematical Fluid Mechanics: to prove existence and
uniqueness of solutions to \eqref{poiseuille2d}-\eqref{poinf} for arbitrary values of $\Fe \in \mathbb{R}$ in such domains, together with a precise description of the asymptotic behavior of the solutions at large distances, see \cite[Chapter XIII]{galdi2011introduction} and \cite{pileckas2007navier}.\par

The best results for the case of arbitrary large fluxes were achieved by Ladyzhenskaya \& Solonnikov \cite{ladyzhenskaya1980finding} in 1979 
(see \cite{ladyzhenskaia1979determination} for a preliminary version and \cite{ladyzhenskaya1983determination} for the English translation to which we refer in the sequel). 
They considered a \textit{junction} of unbounded channels of different shape, both in 2D and 3D, in particular, when the simply connected fluid domain $\Omega$ possesses,
outside some compact region, a finite number $J\geq1$ of rectangular outlets to infinity $\Omega_j$ ($j\in\{1,\dots,J\}$).
See Figure~\ref{fig:ladysol1} for an illustrative example of such 2D region.
\begin{figure}[H]\begin{center}
		\includegraphics[scale=0.7]{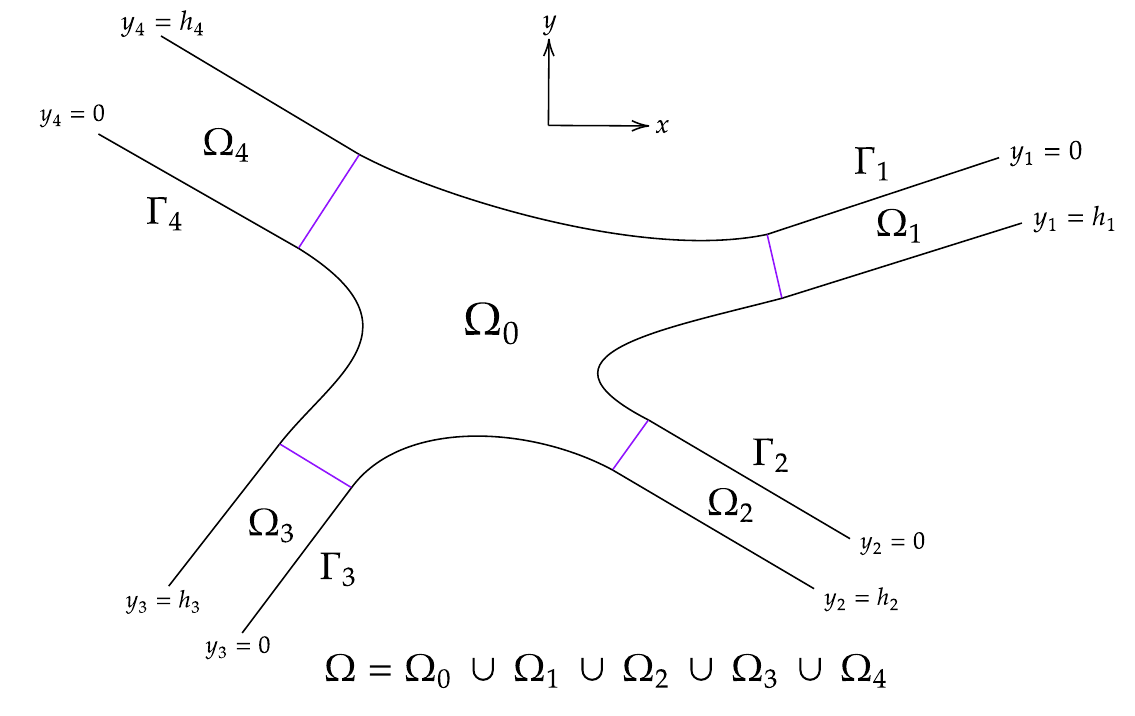}
	\end{center}
	\vspace{-7mm}
	\caption{Junction of unbounded channels considered by Ladyzhenskaya \& Solonnikov, with $J=4$.}
	\label{fig:ladysol1}
\end{figure}
Instead of imposing the asymptotic behavior \eqref{poinf} on each outlet, they suggested to seek smooth solutions
$(\mathbf{u}, p) \in \mathcal{C}^{2}(\overline{\Omega}) \times \mathcal{C}^{1}(\overline{\Omega})$ to \eqref{poiseuille2d} with a
{\it given transverse flux} across each rectangular end, that is, \eqref{poiseuille2d} was supplemented with the Heywood flux conditions
\cite{heywood2006auxiliary,kapitanskii1983spaces}:
\begin{equation} \label{fluxcond1}
\int_{\sigma_j} \mathbf{u} \cdot \ee_1^j = \Fe_j \qquad \forall j \in \{1,\dots,J \} \, ,
\end{equation}
where $\sigma_j \subset \overline{\Omega_{j}}$ denotes any straight cross-section segment  orthogonal to the unit vector $\ee_1^j \in \mathbb{R}^2$
of the outlet $\Omega_{j}$, whereas the flows
$\Fe_j\in\mathbb{R}$ are given and satisfy the \textit{necessary inflow/outflow condition}
\begin{equation} \label{goc1}
\sum_{j=1}^{J} \Fe_j = 0\, .
\end{equation}
In \cite[Theorem 3.1]{ladyzhenskaya1983determination}, the Authors prove the existence of a solution to \eqref{poiseuille2d}-\eqref{fluxcond1}
for any $\Fe_1,...,\Fe_J \in \mathbb{R}$ satisfying \eqref{goc1}
and they show that the Dirichlet integral of any solution remains uniformly bounded in every portion of each outlet having unit length,
as well as the linear growth of such quantity along each outlet. Such estimates on the Dirichlet norm are used in
\cite[Section 2]{ladyzhenskaya1983determination} to prove unique solvability of \eqref{poiseuille2d}-\eqref{fluxcond1} and the asymptotic
behavior \eqref{poinf} at infinity, under a {\it smallness assumption on the fluxes}. Ensuring this asymptotic behavior in the far field
for arbitrarily large fluxes constitutes nowadays the so-called
\textit{Leray--Ladyzhenskaya problem}, see the recent contributions \cite{wang2022existence,wang2022uniform}. A deep consequence of the works by Ladyzhenkaya \& Solonnikov is the realization that the global uniqueness of the Poiseuille flow \eqref{eq:Poi} in a strip implies the unrestricted solvability of the original Leray problem \eqref{poiseuille2d}-\eqref{poinf} \footnote{In this sense, we quote a rather \textit{mysterious} statement from \cite[page 734]{ladyzhenskaya1983determination}: “It seems to us appropriate
to formulate it [Leray's problem] in a broader manner since for large flows, even in cylindrical pipes, one
observes solutions which are different from the Poiseuille flows''.}.

The main goal of the present article is to further generalize and extend the results of Ladyzhenskaya \& Solonnikov,
by considering a non-simply connected junction of unbounded channels and inhomogeneous Dirichlet boundary conditions (thereby allowing for the presence of interior sinks and sources
in a compact region). In particular, under these conditions the Poiseuille flow \eqref{Pflow} has to
be replaced by a {\it Couette-Poiseuille flow}, see \eqref{CP} below. Under a general outflow constraint, we prove the existence of a solution with a
uniformly bounded Dirichlet integral in every compact subset, see Theorem \ref{th:main}. For small data of the problem, we also prove unique solvability and attainability of Couette-Poiseuille flows at infinity, see, respectively, Theorem \ref{th:main2} and Theorem \ref{th:main3}. 

For compactly supported boundary data $\textbf{a}$ with sufficiently small fluxes across the obstacles $\Sigma_i$, the existence assertion of
Theorem~\ref{th:main} can be deduced from \cite{kaulakyte2012nonhomogeneous}. We also highlight previous works by Neustupa \cite{neustupa2009steady,neustupa2010new} where an additional extension assumption on $\textbf{a}$ was imposed. Moreover, for the plane symmetric case with one outlet at infinity
($J=1$), the existence assertion of Theorem~\ref{th:main} follows from the~results in \cite{chipot2017nonhomogeneous}. However, the methods of these
papers do not allow to cover the~considered case of general boundary data~$\ae$ with large fluxes.

In the proofs, we use a rather nontrivial combination of the Ladyzhenskaya--Solonnikov approach from~\cite{ladyzhenskaya1983determination} (where the homogeneous case $\ae\equiv\0$ was considered) with Leray's  \textit{reductio ad absurdum} argument, which is used in some original way in order to prove an analog of the~Leray--Hopf inequality. Note that  
Leray's  contradiction argument deals with limiting Euler solution, which is studied by 
subtle methods of real analysis developed in~\cite{BKK13} and applied previously for the Leray problem in planar bounded domains in~\cite{kpr,korobkov2015solution,korobkov2024steady}. 
We refer to Section \ref{glance} for a brief description of the main novelties and ideas involved in the proofs of our results.

\pagebreak

\section{Main results}

For every $r>0$ let $B_{r} \subset \R^2$ be the open disk of radius $r$, centered at the origin. We introduce

\begin{definition} \label{admissibledom}
Given two integers $I \geq 0$ and $J \geq 1$, a~set $\Omega\subset\R^2$ is called \textbf{admissible} if it is open, unbounded,  connected with a boundary of class $\mathcal{C}^2$, and $\Omega$ can be decomposed as the disjoint union of $J+1$ regions in the following way:
\begin{equation} \label{defomega}
	\Omega = \Omega_0 \cup \biggl(\bigcup_{j=1}^J \Omega_j\biggr)\, ,
\end{equation}
where each $\Omega_j \subset \R^2$, $j \in \{1,\dots,J \}$, is an open \textit{rectangular outlet towards infinity} so that, up to some
rigid transformation $z \in \mathbb{R}^2 \mapsto (x_j(z), y_j(z))  \in \mathbb{R}^2$ and for some $h_{j} > 0$, it can be represented as
\begin{equation}\label{Omegaj}
\Omega_j = \{ z \in\R^2 \ | \ x_j(z)>0 \, , \quad 0 < y_j(z) < h_j \ \} \qquad \forall j \in \{1,\dots,J \} \, .
\end{equation}
In turn, $\Omega_0 \subset \R^2$ is a~bounded and {\rm(}possibly{\rm)} non-simply connected set
$$
\Omega_0 \doteq \Omega \setminus \biggl(\bigcup_{j=1}^J \Omega_j\biggr) \subset B_{R}\qquad \mbox{for some }R>0
$$
and $\Omega_{0}$ surrounds several compact sets $\Sigma_i$, $i\in\{1,\dots,I\}$,  called the \textbf{obstacles}, with boundaries of class~$\mathcal{C}^2$.
\end{definition}	
In Figure \ref{fig:4} we depict an example of an admissible domain $\Omega$, according to Definition \ref{admissibledom}. We refer to Section~\ref{physics} for two physical models described by admissible domains.
\begin{figure}[H]\begin{center}
		\includegraphics[scale=0.7]{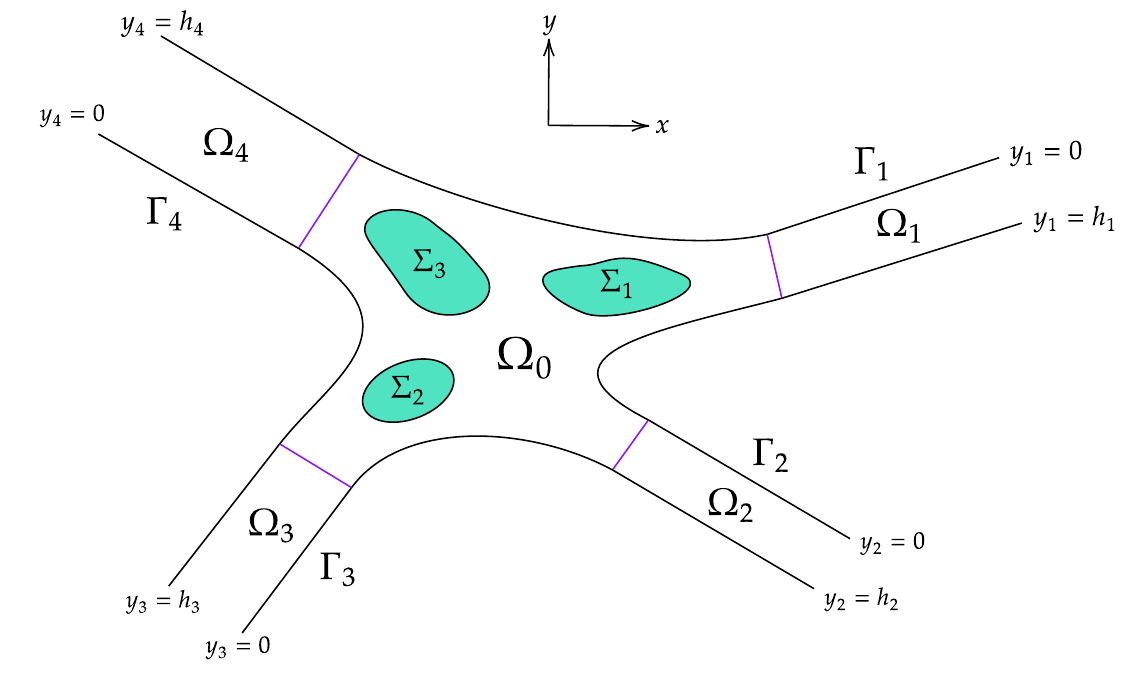}
	\end{center}
	\vspace{-7mm}
	\caption{An admissible domain $\Omega$ with $I=3$ and $J=4$. }
	\label{fig:4}
\end{figure}

Therefore, for every $j \in \{1,\dots,J \}$, the boundary of $\Omega_j$ consists of three components, that is,
\begin{equation}\label{bdd0}
	\partial\Omega_j=\sigma_j^0\cup\GG_j^0\cup\GG_j^1 \, ,
\end{equation}
where, in local coordinates,
$$
\begin{array}{c}
\sigma_j^0 \doteq \{ z \in\R^2 \ | \ x_j(z)=0 \, , \quad 0 < y_j(z) < h_j \, \} \, , \\[6pt]
\GG_j^0 \doteq \{ z \in\R^2 \ | \ x_j(z) \geq 0 \, , \quad y_j(z)= 0 \, \} \, ,\qquad
\GG_j^1 \doteq \{ z \in\R^2 \ | \ x_j(z) \geq 0 \, , \quad y_j(z)= h_{j} \, \} \, ,
\end{array}
$$
see Figure \ref{fig:1} for an~illustration.
\begin{figure}[H]\begin{center}
\includegraphics[scale=0.7]{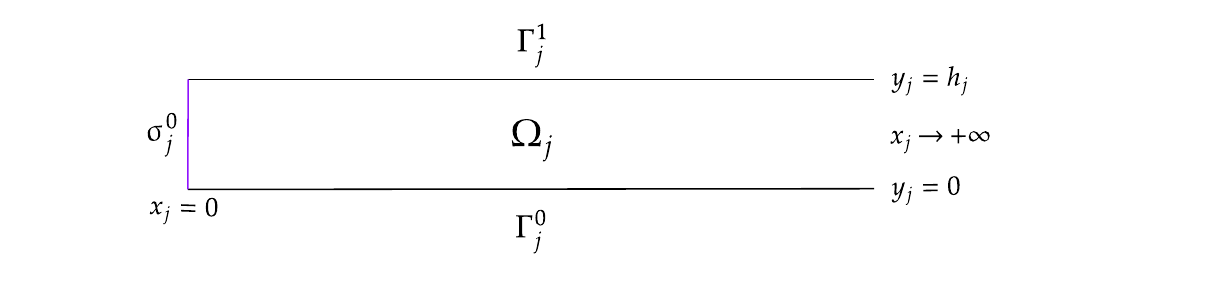}
\end{center}
\vspace{-5mm}
\caption{Decomposition of the boundary of an outlet $\Omega_{j}$, for $j \in \{1,\dots,J \}$.}\label{fig:1}
\end{figure}
\noindent
Moreover, the whole boundary of $\Omega$ consists of $I+J$ disjoint connected components: $I$ bounded closed curves
$\partial \Sigma_i \subset \partial \Omega_0$, for $i \in \{1,\dots,I \}$, and $J$ unbounded lines $\Gamma_{j}$, for $j \in \{1,\dots,J \}$. We then write
$$
\partial\Omega = \Gamma\cup\partial \Sigma  \doteq \biggl(\bigcup\limits_{j=1}^J\Gamma_j\biggr) \cup \biggl(\bigcup\limits_{i=1}^I \partial \Sigma_i\biggr),
$$
and adopt the following convention agreements (compare with \eqref{bdd0}):
\begin{equation}\label{bdd1}
\Gamma_{J+1}=\Gamma_1\, , \quad \text{and} \quad \Gamma_j^0\subset\Gamma_j \, , \quad \Gamma_j^1\subset\Gamma_{j+1} \qquad \forall j \in \{1,\dots,J \} \, .
\end{equation}

Let $\fe\in L_\loc^{2}(\overline\Omega)$ be an external force acting on the fluid and satisfying the uniform bound
\begin{equation}  \label{uef}
	\int_{\Omega \cap B(z_0)}|\mathbf{f}|^2\le C_{*} \qquad \forall z_0\in\R^2\, ,
\end{equation}
where $B(z_0) \subset \mathbb{R}^2$ is the unit disc centered at $z_0\in\R^2$, for some $C_{*}>0$ that is independent of $z_{0}$.
We consider a boundary velocity $\ae\in W^{3/2,2}_\loc(\partial\Omega)$ such that
\begin{equation}  \label{bda}
\ae(z)=\be_j^0\quad\forall z\in\Gamma_j^0 \, , \qquad \ae(z)=\be_j^1\quad\forall z\in\Gamma_j^1\, ,\qquad \forall j \in \{1,\dots,J \} \, ,
\end{equation}
where $\be_{j}^m\in\R^2$ are constant vectors obeying an {\it impermeability condition} on $\Gamma \setminus \partial \overline{\Omega_0}$, that is,
\begin{equation}  \label{bda-1}
	\be_{j}^m\cdot\n=0 \qquad \forall m \in \{0,1 \} \, , \ \forall j \in \{1,\dots,J \} \, .
\end{equation}
Hereafter, $\n \in\R^2$ denotes the outward unit normal to $\partial\Omega$ (exiting from $\Gamma_j^0$ and $\Gamma_j^1$ for $j \in \{1,\dots,J \}$,
entering in $\Sigma_{i}$ for $i \in \{1,\dots,I \}$). In local coordinates $(x_j,y_j)$ we then have
\begin{equation} \label{slipbcori}
\forall m\in\{0,1\}\, \ \forall j\in\{1,\dots,J \}\qquad\exists\, b_j^m\in\mathbb{R}\quad\mbox{ such that }\quad\be_j^m = (b_j^m, 0)\, .
\end{equation}
Furthermore, let $\Fe_j\in\R$ ($j \in \{1,\dots,J \}$) be given and such that
\begin{equation} \label{ch5}
\int_{\partial\Omega} \ae\cdot\n+\sum\limits_{j=1}^J\Fe_j = 0 \, .
\end{equation}
The central theme of this work is the following boundary-value problem associated to the steady-state Navier-Stokes equations in $\Omega$:
\begin{equation}  \label{NSE}
	\left\{
	\begin{aligned}
		&- \Delta \mathbf{u} + (\mathbf{u} \cdot \nabla) \mathbf{u} + \nabla p = \mathbf{f} \, , \quad \nabla \cdot \mathbf{u} = 0 \ \ \mbox{ in } \ \ \Omega \, , \\[4pt]
		&\mathbf{u} = \ae \ \ \mbox{ on } \ \ \partial \Omega \, , \\[4pt]
		&  	\int_{\sigma_j^0} \mathbf{u} \cdot \ee_1^j = \Fe_j \qquad \forall j \in \{1,\dots,J \} \, .
	\end{aligned}
	\right.
\end{equation}
Identity \eqref{NSE}$_2$
prescribes the boundary velocity $\ae$ on $\partial \Omega$ (inhomogeneous Dirichlet boundary conditions), while \eqref{NSE}$_3$ imposes that the
transverse flux along each outlet $\Omega_j$ equals $\Fe_j\in\R$, for every $j \in \{1,\dots,J \}$. 
Notice that, in analogy to \eqref{goc1}, only the general outflow condition \eqref{ch5} is assumed. Moreover, due to \eqref{bda-1} we have
\begin{equation} \label{fluxa}
	\int_{\partial\Omega} \ae\cdot\n = \int_{\partial\Omega\, \cap\,\partial \Omega_0} \ae\cdot\n =  \sum_{i=1}^{I} \int_{\partial\Sigma_{i}} \ae\cdot\n  +\sum_{j=1}^{J}  \int_{ \Gamma_{j}\cap \partial\Omega_{0} } \ae\cdot\n  \, .
\end{equation}

The first main result of this article reads:

\begin{theorem}[{\bf Existence}]\label{th:main}
Let $\Omega \subset \mathbb{R}^2$ be an admissible domain. Suppose that $\fe\in L_\loc^{2}(\overline\Omega)$ is an external force verifying
\eqref{uef}, $\ae\in W^{3/2,2}_\loc(\partial\Omega)$ is a boundary velocity obeying \eqref{bda}-\eqref{bda-1}, and $\Fe_1,...,\Fe_J\,{\in}\,\mathbb{R}$
satisfying the compatibility condition \eqref{ch5}. Then, the system \eqref{NSE} admits at least one strong solution
$(\ue,p)\in W^{2,2}_\loc(\Omega)\times W^{1,2}_\loc(\Omega)$ satisfying the uniform bound
\begin{equation}  \label{NSE-est}
\|\ue\|_{W^{2,2}(\Omega\cap B(z_0))}+\|\nabla p\|_{L^2(\Omega\cap B(z_0))}\le C_* \qquad \forall z_0\in\R^2\, ,
\end{equation}
where $C_* > 0$ depends on $\Omega$, $\fe$, $\ae$ and $\{\mathcal{F}_j\}^{J}_{j=1}$, but not on $z_0$.
\end{theorem}

\begin{remark}{\rm 
Theorem \ref{th:main} holds under the sole outflow condition \eqref{ch5}, while no assumptions on the size of the data of the problem are required.
The boundary velocity $\ae$ is required to be tangential only on the boundaries of the channels, see (\ref{bda})--(\ref{bda-1}), and not on $\partial\Omega_0$. Moreover, the external force is only required to satisfy $\fe\in L_\loc^{2}(\overline\Omega)$ and \eqref{uef} which allows, in particular, for periodic forces in the channels $\Omega_j$.
For the validity of Theorems \ref{th:main2} and \ref{th:main3} below, we need the stronger assumption that $\fe\in L^{2}(\Omega)$.}
\end{remark}

The possibility of an inhomogeneous boundary velocity $\ae\in W^{3/2,2}_\loc(\partial\Omega)$ (see \eqref{bda}-\eqref{bda-1})
suggests that the expected limiting condition \eqref{poinf} (in the case of small fluxes) should be replaced with a convergence to
\textit{Couette-Poiseuille flows} in each rectangular outlet. Explicitly, in local coordinates we put
\begin{equation} \label{CP}
\mathbf{CP}_j(x_{j},y_{j})\doteq(A_jy_j^2+B_jy_j+b_j^0,0)\, ,\quad\Pi_{j}(x_{j}) \doteq2A_j\, x_{j}\qquad\forall(x_j,y_j)\in\overline{\Omega_{j}}\, ,
\quad\forall j \in \{1,\dots,J \} \, ,
\end{equation}
with
\begin{equation} \label{CPconstants}
A_j \doteq  \frac{3(b^0_j+b^1_j)h_j - 6 \mathcal{F}_j}{h_j^3} \, , \quad B_j \doteq \frac{-(4b_j^0+2b_j^1)h_j + 6 \mathcal{F}_j}{h_j^2} \, , \qquad \forall j \in \{1,\dots,J \} \, .
\end{equation}
In particular, notice that $\mathbf{CP}_j(x_{j},0) = \be_j^0$ and $\mathbf{CP}_j(x_{j}, h_{j}) = \be_j^1$ for every $x_{j} \geq 0$, and also
\begin{equation}  \label{fluxcp0}
\int_{\sigma_j^0} \mathbf{CP}_j \cdot \ee_1^j = \int_{0}^{h_{j}} (A_j y_j^2 + B_j y_j + b_j^0) \, dy_{j} = \mathcal{F}_j \qquad \forall j \in \{1,\dots,J \} \, .
\end{equation}
The pair $(\mathbf{CP}_{j},\Pi_{j})$ solves, in classical sense, the following stationary Navier-Stokes system in $\Omega_{j}$:
\begin{equation} \label{nscouette}
	\left\{
	\begin{aligned}
		&- \Delta \mathbf{CP}_{j} + (\mathbf{CP}_{j} \cdot \nabla) \mathbf{CP}_{j} + \nabla \Pi_{j}= \0 \, , \quad \nabla \cdot \mathbf{CP}_{j} = 0 \ \ \mbox{ in } \ \ \Omega_{j} \, , \\[4pt]
		& \mathbf{CP}_{j} = \be_j^0 \ \ \mbox{ on } \ \ \Gamma_{j}^{0} \, , \qquad \mathbf{CP}_{j} = \be_j^1 \ \ \mbox{ on } \ \ \Gamma_{j}^{1}.
	\end{aligned}
	\right.
\end{equation}

Under suitable smallness conditions on the data of the problem, we obtain an asymptotic description of the behavior of solutions to
\eqref{NSE} along each rectangular outlet. Precisely, the second main result of this article reads:

\begin{theorem}[{\bf Asymptotics}]\label{th:main2} Let $\Omega \subset \mathbb{R}^2$ be an admissible domain. Suppose that $\fe\in L^{2}(\Omega)$
is an external force, $\ae\in W^{3/2,2}_\loc(\partial\Omega)$ is a boundary velocity obeying \eqref{bda}-\eqref{bda-1}, and
$\Fe_1,...,\Fe_J \in \mathbb{R}$ satisfy the compatibility condition \eqref{ch5}. There exists $\delta_{*} > 0$,
depending only on $\Omega$, such that, if
\begin{equation}  \label{uni-1}
\max\limits_{j \in \{1, \dots, J\}} \left(|\be_j^1-\be_j^0|  + |\mathcal{F}_j-b_j^0 \, h_j| \right) <\delta_* \, ,
\end{equation}
then any strong solution $(\ue,p)\in W^{2,2}_\loc(\Omega)\times W^{1,2}_\loc(\Omega)$ to \eqref{NSE} satisfying~\eqref{NSE-est} admits the bound
\begin{equation}  \label{uni-2}
	\|\ue - \mathbf{CP}_j\|_{W^{2,2}(\Omega_j)}+\|\nabla p-\nabla\Pi_j\|_{L^2(\Omega_j)}< \infty \qquad \forall j \in \{1,\dots,J \} \, .	
\end{equation}
Moreover, 
\begin{equation} \label{convergencecp}
\ue \longrightarrow \mathbf{CP}_j \quad \text{uniformly as} \quad |x_j| \to \infty \quad \text{in} \quad \Omega_{j} \, , \ \ \forall j \in \{1,\dots,J \} \, .
\end{equation}
\end{theorem}

The third main result of this article guarantees, in an analogous smallness regime, the unique solvability of problem \eqref{NSE}
within the class of solutions veryfing the estimate \eqref{NSE-est}.
	
\begin{theorem}[{\bf Uniqueness}]\label{th:main3} Let $\Omega \subset \mathbb{R}^2$ be an admissible domain.
Suppose that $\fe\in L^{2}(\Omega)$ is an external force, $\ae\in W^{3/2,2}_\loc(\partial\Omega)$ is a boundary velocity obeying
\eqref{bda}-\eqref{bda-1}, and $\Fe_1,...,\Fe_J \in \mathbb{R}$ satisfy the compatibility condition \eqref{ch5}.
There exists $\delta_{0} > 0$, depending only on $\Omega$, such that, if
\begin{equation}  \label{uni-4}
\|\fe\|_{L^2(\Omega)}+\|\ae\|_{W^{3/2,2}(\partial\Omega \cap B_{R})}+
\max\limits_{j \in \{1, \dots, J\}}\left( |\be_j^0|+|\be_j^1|+ |\Fe_j| \right)< \delta_0 \, ,
\end{equation}
then the system \eqref{NSE} admits a unique strong solution $(\ue,p)\in W^{2,2}_\loc(\Omega)\times W^{1,2}_\loc(\Omega)$ satisfying \begin{equation} 
	\sup_{z_0 \in \mathbb{R}^2}\|\ue\|_{W^{2,2}(\Omega\cap B(z_0))}< \infty\,.
\end{equation}
\end{theorem}

\begin{remark}
{\rm On the one hand, \eqref{uni-4} imposes a smallness condition on all the data of the problem: external force, boundary velocity and fluxes. On the other hand, \eqref{uni-1} represents a compatibility condition as it dictates, 
roughly speaking, that {\it the gradients} of the corresponding  
Couette-Poiseuille flows $\mathbf{CP}_j$ from~(\ref{fluxcp0})--(\ref{nscouette}) are 
 uniformly small.}
\end{remark}

\section{Proof of Theorem \ref{th:main}}

\subsection{The proof at a glance} \label{glance}
Among other issues, the proofs of Theorems \ref{th:main}--\ref{th:main2}--\ref{th:main3} are obtained through the following steps.\par
- In Section \ref{sec:2.1} we build a special \textit{flux carrier} of the boundary data in \eqref{NSE}$_2$, that is, a solenoidal extension of the boundary
velocity which, moreover, verifies the flux condition \eqref{NSE}$_3$. Such flux carrier is the solution of a suitable Stokes problem in $\Omega_{0}$, and satisfies a \textit{uniform} Leray-Hopf inequality in every outlet. This partition of $\Omega$ is a necessary step, in view of the fact that in multiply connected domains one cannot construct
Hopf extensions for general boundary data, see \cite{heywood2011impossibility,takeshita1993remark}. \par
- In Section \ref{proofs-sec} we adapt to our setting the celebrated method of \textit{invading domains};   under the~sole zero total flux condition we are able to find a sequence of solutions
$\{(\ue_k,p_k) \}_{k \in \mathbb{N}}$ to the Navier-Stokes system in any truncated domain because of Koroblov--Pileckas--Russo general existence theorem, see \cite{korobkov2015solution,korobkov2024steady}. \par
- The ultimate step in the proof of Theorem \ref{th:main} is to show that, as $k\to+\infty$ and up the extraction of a subsequence,
the sequence $\{(\ue_k,p_k) \}_{k \in \mathbb{N}}$ converges locally (in a sense to be made precise later, see Section \ref{proofs-sec}) to a solution $(\ue,p)$ of the
original problem \eqref{NSE} which, moreover, satisfies the required estimate \eqref{NSE-est}. \par
- To prove the required estimate, the main tool in the~Ladyzhenskaya--Solonnikov paper \cite{ladyzhenskaya1983determination}  
 was the classical Leray--Hopf inequality (see, e.g., (\ref{gg10})\,) obtained by the standard Hopf cut-off procedure under homogeneous boundary conditions. In our case, under general boundary data assumptions, this cut-off procedure does not work anymore, and one cannot expect the Leray--Hopf inequality in the~general case, see \cite{heywood2011impossibility,takeshita1993remark}. Nevertheless, fortunately, the Leray--Hopf inequality still can be proved {\it for large $k$} based on a subtle Leray's  \textit{reductio ad absurdum} argument (traced back to the seminal work of Leray~\cite{leray1933etude}). The cornerstone here is the triviality of the limiting Euler solution~(see Appendix~\ref{poet2}\,) proved using the Korobkov--Pileckas--Russo approach from~\cite{kpr,korobkov2015solution,korobkov2024steady}. In turn, subtle methods of real analysis (Morse-Sard type theorems and fine properties of the level sets of Sobolev functions) developed by Bourgain et al.~\cite{BKK13} are very useful here as well.  Compared with the setting of \cite{korobkov2015solution}, the new challenge here is due to the unboundedness of the domain $\Omega$ and the linear growth of total Dirichlet energies of the invading domain solutions. 
 \par
- In Sections \ref{poet4}--\ref{poet6}--\ref{poet7} we bound the growth of the Dirichlet integral of the solutions $\{(\ue_k,p_k) \}_{k \in \mathbb{N}}$ on the sequence of invading domains:
this is done in three steps, first the global linear growth, then the uniform linear growth, finally through uniform local estimates. Here, at every step, we use 
some elegant real-analysis techniques from~\cite{ladyzhenskaya1983determination}  (see Sections~\ref{poet6}--\ref{poet7}).  We are
so able to prove the uniform bound \eqref{NSE-est} in Section \ref{conclusionproof}.\par
- Finally the uniqueness and asymptotic behavior for small data are obtained in the last Sections~\ref{poet8}--\ref{poet8.2} again using estimates from~\cite{ladyzhenskaya1983determination}.

\subsection{Construction of a flux carrier} \label{sec:2.1}
Let $\Omega \subset \mathbb{R}^2$ be an admissible domain in the sense of Definition \ref{admissibledom}.
In agreement with \eqref{Omegaj},
for all $j\in\{1,\dots,J\}$ and $\ell> 0$, we set
\begin{equation}\label{bdd00flux}
\sigma_j^\ell\doteq\{z\in\Omega_j\, |\ x_j=\ell\}\, .
\end{equation}
Given a solution $(\mathbf{u}, p)$ to \eqref{NSE},
the Divergence Theorem, together with assumptions \eqref{bda}-\eqref{bda-1} and the incompressibility condition \eqref{NSE}$_1$, implies that
$$
\int_{\sigma_j^\ell}  \mathbf{u} \cdot \ee_1^j = \Fe_j \qquad \forall j \in \{1,\dots,J \} \ \ \forall \ell>0 \, .
$$
Given two extended real numbers $0 \leq a < b \leq \infty$, we define, as in \eqref{bdd00flux}, the following sets:
\begin{equation} \label{Omega_ab}
	\Omega^{a,b}_{j} \doteq \{ z \in\Omega_j\, | \ a<x_j< b\} \, , \qquad
	\Omega^{a,b} \doteq  \bigcup\limits_{j=1}^{J} \Omega_j^{a,b},\qquad\OO^{a} \doteq \Omega_0 \cup \Omega^{0,a} \, .
\end{equation}
In Figure \ref{fig:5} we depict an example of the truncated domain $\Omega^2$.
\begin{figure}[H]
\begin{center}
		\includegraphics[scale=0.75]{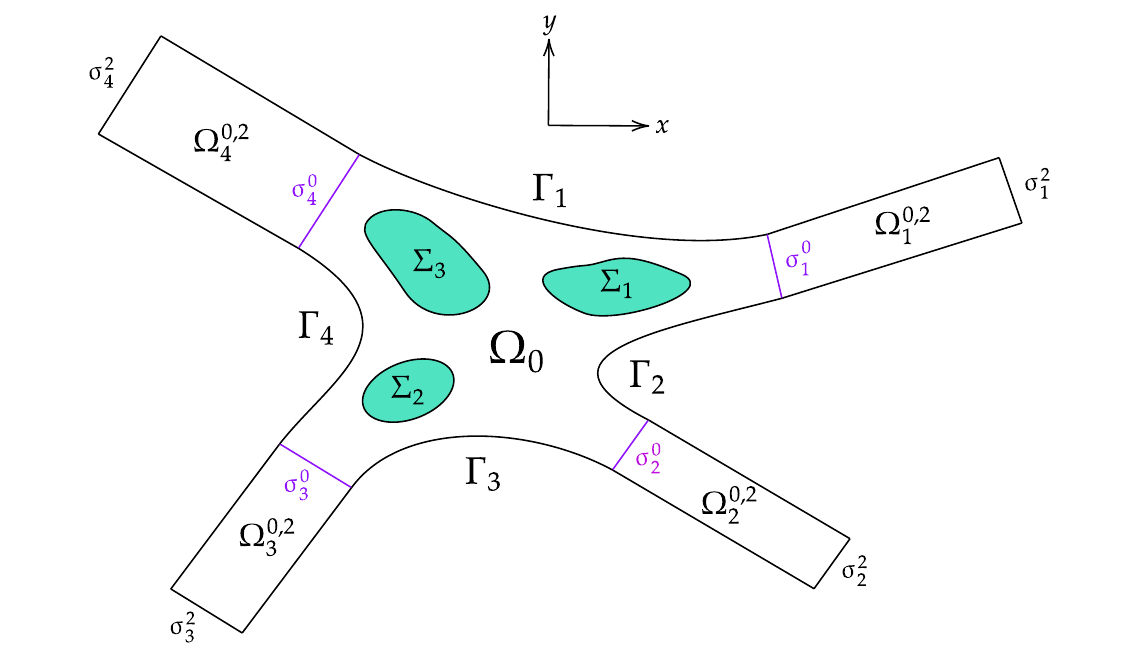}
	\end{center}
	\vspace{-5mm}
	\caption{The truncated domain $\Omega^2$, defined in~(\ref{Omega_ab}), with $I=3$, \ $J=4$ and $\Gamma_5 = \Gamma_1$.}
\label{fig:5}
\end{figure}

In this section we state the existence of a solenoidal vector field in $\Omega$ which, all together, satisfies the inhomogeneous boundary conditions \eqref{NSE}$_2$ on
$\partial \Omega$, the flux constraints \eqref{NSE}$_3$ and a \textit{uniform} Leray-Hopf inequality (see, e.g., \cite{galdi1991existence}\,) in {\it each bounded part of every outlet}.
More precisely, we prove:

\begin{theorem}\label{fluxcarrier}
Let $\Omega \subset \mathbb{R}^2$ be an admissible domain. Suppose that $\ae\in W^{3/2,2}_\loc(\partial\Omega)$ is a boundary velocity obeying
\eqref{bda}-\eqref{bda-1}, and $\Fe_1,...,\Fe_J \in \mathbb{R}$ satisfy the compatibility condition \eqref{ch5}. Then, for any $\e \in (0,1)$,
there exists a vector field $\mathbf{U} \in W^{2,2}_{\loc}(\overline\Omega)$ such that
\begin{equation} \label{vecpsi}
 \nabla \cdot \mathbf{U} =0 \ \ \mbox{ in } \ \ \Omega \, ; \qquad \mathbf{U} = \mathbf{a} \ \ \mbox{ on } \ \ \partial \Omega \, ; \qquad \int_{\sigma_j^0} \mathbf{U} \cdot \ee_1^j = \Fe_j \qquad \forall j \in \{1,\dots,J \} \, .
\end{equation}
Moreover, we have the local uniform bound
\begin{equation}  \label{uboundz2}
	\| \mathbf{U} \|_{W^{2,2}(\Omega\cap B(z_0))}  \le C_* \qquad \forall z_0\in\R^2\, ,
\end{equation}
where $C_* > 0$ is a constant depending on $\Omega$, $\ae$, $\{\mathcal{F}_j\}^{J}_{j=1}$ and $\varepsilon$, but is independent of $z_0 \in\R^2$.
Finally, given $j \in \{1,\dots, J\}$, $2\le a<b\le\infty$, and any vector field $\bbe \in W^{1,2}\bigl( \Omega_j^{a,b}\bigr)$ such that
  \begin{equation}\label{testeta}
\nabla \cdot \bbe =0 \ \ \mbox{ in } \ \ \Omega_j^{a,b} \qquad \text{and} \qquad  \bbe = 0 \ \ \mbox{ on } \ \ \Gamma \cap \partial  \Omega_j^{a,b} \, ,
\end{equation}
the Leray-Hopf inequality holds
\begin{equation}\label{HII}
\left| \int_{  \Omega_j^{a,b}} (\bbe\cdot \nabla)\mathbf{U} \cdot \bbe \right| +
\left| \int_{ \Omega_j^{a,b}} (\bbe\cdot \nabla)\bbe\cdot\mathbf{U} \right| \leq \varepsilon \,
\| \nabla\bbe \|_{L^{2}(  \Omega_j^{a,b})}^2 \, .
\end{equation}
\end{theorem}
It is important that the estimate \eqref{HII} is independent of the length of the domains $\Omega_j^{a,b}$. The proof of Theorem~\ref{fluxcarrier} is based on the standard Hopf cut-off procedure at any outlet (in a translation invariant setting, i.e., $\mathbf{U}$ depends on the vertical variable~$y_j$ in any outlet) as well as a~gluing argument, and we postpone it until~Appendix~\ref{appendix}. For our purposes it is enough to take $\e=\frac18$, that is, we will use an extension~$\mathbf{U}$ such that 
\begin{equation}\label{HI}
\left| \int_{  \Omega_j^{a,b}} (\bbe\cdot \nabla)\mathbf{U} \cdot \bbe \right| +
\left| \int_{ \Omega_j^{a,b}} (\bbe\cdot \nabla)\bbe\cdot\mathbf{U} \right| \leq \frac18 \,
\| \nabla\bbe \|_{L^{2}(  \Omega_j^{a,b})}^2 
\end{equation}
holds for any vector field $\bbe \in W^{1,2}\bigl( \Omega_j^{a,b}\bigr)$ satisfying~(\ref{testeta}).

\subsection{Leray's method of \textit{invading domains}}\label{proofs-sec}

Let $\Omega \subset \mathbb{R}^2$ be an admissible domain. Suppose that $\fe\in L_\loc^{2}(\overline\Omega)$ is an external force verifying~\eqref{uef}, $\ae\in W^{3/2,2}_\loc(\partial\Omega)$ is a boundary velocity obeying \eqref{bda}-\eqref{bda-1}, and
$\Fe_1,...,\Fe_J \in \mathbb{R}$ satisfy the compatibility condition \eqref{ch5}. Let
$\mathbf{U}\in W^{2,2}_{\loc}(\overline\Omega)$ be a~flux carrier arising from Theorem \ref{fluxcarrier} and satisfying~(\ref{HI}).
Define the vector field $\mathbf{g} \in L_\loc^2(\overline\Omega)$ by
$$
\mathbf{g} \doteq \mathbf{f} +   \Delta \mathbf{U} - (\mathbf{U} \cdot \nabla) \mathbf{U} \, .
$$
The strong solution $(\mathbf{u}, p) \in W_{\text{loc}}^{2,2}(\Omega) \times W_{\text{loc}}^{1,2}(\Omega)$ to \eqref{NSE} will be sought in the
form
\begin{equation}  \label{shift}
\mathbf{u} = \mathbf{U} + \mathbf{w} \, ,
\end{equation}
where the vector field $\mathbf{w} \in W_{\text{loc}}^{2,2}(\Omega)$ satisfies, in strong form,
the following system:
\begin{equation}  \label{NSE-w}
	\left\{
	\begin{aligned}
		&- \Delta \mathbf{w} + (\mathbf{w} \cdot \nabla) \mathbf{w} + (\mathbf{U} \cdot \nabla) \mathbf{w} +(\mathbf{w} \cdot \nabla) \mathbf{U} + \nabla p = \mathbf{g} \, , \quad \nabla \cdot \mathbf{w} = 0 \ \ \mbox{ in } \ \ \Omega \, , \\[4pt]
		&\mathbf{w} = \0 \ \ \mbox{ on } \ \ \partial \Omega \, , \\[4pt]
		&  	\int_{\sigma_j^0} \mathbf{w} \cdot \ee_1^j = 0 \qquad \forall j \in \{1,\dots,J \} \, .
	\end{aligned}
	\right.
\end{equation}
In order to find a strong solution $(\mathbf{w}, p) \in W_{\text{loc}}^{2,2}(\Omega) \times W_{\text{loc}}^{1,2}(\Omega)$ to \eqref{NSE-w}, consider a strictly increasing sequence of bounded domains $\{  \mathcal{B}_k \}_{k \in \mathbb{N}}$ (called \textit{invading domains} in the sequel), each one having a boundary of class $\mathcal{C}^{2}$, such that
\begin{equation}  \label{chp1}
	\Omega^2 \subset \mathcal{B}_0 \qquad \text{and} \qquad \bigcup\limits_{k\in \N}  \mathcal{B}_k=\Omega \, ,
\end{equation}
where $\Omega^2$ is as in~(\ref{Omega_ab}). Moreover, we require that each
$$
S_{j,k} \doteq \Omega_j \cap \partial  \mathcal{B}_k \qquad \forall k \in \mathbb{N} \, , \ \forall j \in \{1,\dots,J \} \, ,
$$
is a $\mathcal{C}^{2}$-arc obtained from the initial arc $S_{j,1}$ by translation along the
$x_j$-axis (see  Figure \ref{fig:2} for a depiction). The precise shapes of the arcs $S_{j,1}$ are arbitrary as long as they are
tangent to the corresponding walls $\Gamma_j^0 \cup \Gamma_j^1$, for $\{1,\dots,J \}$, and have zero curvature at their endpoints in such a way that $\partial \mathcal{B}_k\in \mathcal{C}^2$.
\begin{figure}[H]
	\begin{center}
		\includegraphics[scale=0.75]{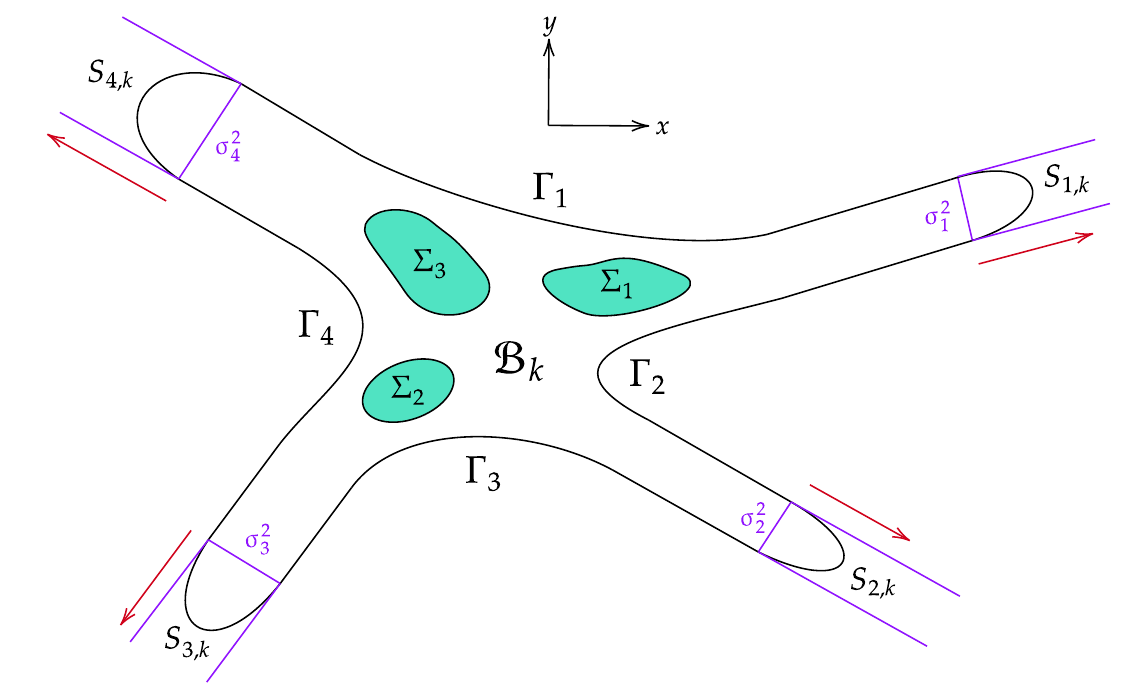}
	\end{center}
	\vspace{-5mm}
	\caption{The \textit{invading domains} procedure.}
	\label{fig:2}
\end{figure}
\noindent
By Theorem~\ref{fluxcarrier}, for any $k\in \mathbb{N}$ we have
$\mathbf{U} \in W^{2,2}(\mathcal{B}_k)$. From \eqref{vecpsi} and the Divergence Theorem we deduce
$$
\int_{S_{j,k}} \mathbf{U} \cdot \mathbf{n} = \Fe_j \qquad \forall j \in \{1,\dots,J \} \, ,
$$
thereby implying that
$$
\int_{\partial  \mathcal{B}_k} \mathbf{U} \cdot \mathbf{n} = \sum_{i=1}^{I}  \int_{\partial\Sigma_{i}} \ae\cdot\n+ \sum_{j=1}^{J} \int_{\partial\Omega_{0} \cap \Gamma_{j}} \ae\cdot\n + \sum_{j=1}^{J} \Fe_j = 0 \, ,
$$
owing to \eqref{ch5} and \eqref{fluxa}. Therefore, \cite[Theorems\,5.1.1\,and\,2.4.3]{korobkov2024steady} (cf. with \cite[Theorem~1.1]{korobkov2015solution})
ensure the existence of a strong solution $(\ue_k, p_k) \in W^{2,2}(\mathcal{B}_k) \times W^{1,2}( \mathcal{B}_k)$ to the following
boundary-value problem:
\begin{equation}  \label{NSE-k}
\left\{
\begin{aligned}
 &- \Delta \mathbf{u}_k + (\mathbf{u}_k \cdot \nabla) \mathbf{u}_k + \nabla p_k = \mathbf{f} \quad\mbox{ in } \  \mathcal{B}_k \, , \quad \nabla \cdot \mathbf{u}_k = 0 \ \ \mbox{ in } \ \  \mathcal{B}_k \, , \\[4pt]
 &\mathbf{u}_k = \mathbf{U} \ \ \mbox{ on } \ \ \partial  \mathcal{B}_k \, .
 \end{aligned}
\right.
\end{equation}
As in \eqref{NSE-w}, given $k \in \mathbb{N}$ we set $$\mathbf{u}_{k}\doteq \mathbf{U} + \mathbf{w}_{k},$$ and the vector field $\mathbf{w}_{k} \in W^{2,2}(\mathcal{B}_k) \cap W_{0}^{1,2}(\mathcal{B}_k)$ must then satisfy, in strong form, the following system:
\begin{equation}  \label{NSE-wk}
\left\{
\begin{aligned}
&- \Delta \mathbf{w}_k + (\mathbf{w}_k \cdot \nabla) \mathbf{w}_k + (\mathbf{U} \cdot \nabla) \mathbf{w}_k +(\mathbf{w}_k \cdot \nabla) \mathbf{U} + \nabla p_k = \mathbf{g}\qquad\mbox{ in }\ \ \mathcal{B}_k, \\
& \nabla \cdot \mathbf{w}_k = 0\qquad\mbox{ in }\  \ \mathcal{B}_k, \\
&\mathbf{w}_k = \0 \qquad\mbox{ on }\ \ \partial  \mathcal{B}_k,
\end{aligned}
\right.
\end{equation}
where $\mathbf{g}=\mathbf{f}+\Delta\mathbf{U}-(\mathbf{U}\cdot\nabla)\mathbf{U}\in L^2(\mathcal{B}_k)$. In particular, setting
$$
W^{1,2}_{0,\sigma}(\mathcal{B}_k) \doteq \left\lbrace \mathbf{v} \in W_{0}^{1,2}(\mathcal{B}_k) \ \big| \ \nabla \cdot \mathbf{v} = 0 \ \ \mbox{in} \ \ \mathcal{B}_k \right\rbrace \, ,
$$
it follows that $\mathbf{w}_{k} \in W_{0,\sigma}^{1,2}(\mathcal{B}_k)$ is a \textit{weak solution} to \eqref{NSE-wk}, meaning that
\begin{equation} \label{weaksolutionsh}
	\int_{\mathcal{B}_k} \nabla \mathbf{w}_{k} \cdot \nabla \bbe + \int_{\mathcal{B}_k} (\mathbf{w}_{k} \cdot \nabla)\mathbf{w}_{k} \cdot \bbe +
	\int_{\mathcal{B}_k} (\mathbf{U} \cdot \nabla) \mathbf{w}_{k} \cdot \bbe + \int_{\mathcal{B}_k} (\mathbf{w}_{k} \cdot \nabla) \mathbf{U}\cdot
	\bbe = \int_{\mathcal{B}_k} \mathbf{g} \cdot \bbe
\end{equation}
for all $\bbe \in W^{1,2}_{0,\sigma}(\mathcal{B}_k)$. Let us define
\begin{equation}  \label{pp6}
	J_k \doteq \| \nabla \mathbf{w}_{k} \|_{L^{2}(\mathcal{B}_k)} \, ,\qquad \widehat{\mathbf{w}}_k\doteq\frac1{J_k} \, \mathbf{w}_{k} \, , \qquad \hp_k\doteq\frac1{J^2_k} \, p_k \qquad \forall k \in \mathbb{N} \, .
\end{equation}
We extend $\widehat{\mathbf{w}}_k$ trivially by zero to the whole $\Omega$ (keeping the same notation), so that
$(\widehat{\mathbf{w}}_{k})_{k  \in \mathbb{N}}  \subset W_{0,\sigma}^{1,2}(\Omega)$.

After setting \eqref{pp6}, we let $k \to \infty$. Two situations may arise.
\newline

\underline{Case (I):} if
\begin{equation}  \label{suppose1}
\liminf\limits_{k\to \infty} J_k < +\infty \, ,
\end{equation}
then the proof of Theorem~\ref{th:main} can be finished almost immediately. Indeed, since the Poincaré inequality holds in $\Omega$, along a~subsequence, $(\mathbf{w}_{k})_{k \in \mathbb{N}}$ is uniformly bounded with respect to $k \in \mathbb{N}$, ensuring the
existence of a vector field $\mathbf{w} \in W_{0,\sigma}^{1,2}(\Omega)$ such that
\begin{equation} \label{wksconv}
	\mathbf{w}_{k} \rightharpoonup \mathbf{w} \ \ \ \text{weakly in} \ W_{0,\sigma}^{1,2}(\Omega)  \ \ \ \text{as} \ \ k \to +\infty \, .
\end{equation}
Since $\we_k$ are solutions to~(\ref{NSE-wk}) in domains $\mathcal{B}_k$, 
the limit $\we$ is a weak solution to (\ref{NSE-w}) with 
\begin{equation} \label{d-b-w}
	\int_\Omega|\nabla\we|^2<\infty.
\end{equation}
Furthermore, since $\mathbf{U} \in W^{2,2}_{\loc}(\overline\Omega)$, $\mathbf{g} \in L_\loc^2(\overline\Omega)$, and $\partial \Omega$ is
of class $\mathcal{C}^{2}$, the regularity theory and standard estimates for the stationary Navier-Stokes system (see, e.g., \cite[Section 3.2.4]{korobkov2024steady})  can
be applied to deduce that $\mathbf{w}\in W_{\text{loc}}^{2,2}(\Omega)$, there exists a corresponding pressure function~$p\in W_\loc^{1,2}(\Omega)$, and the uniform estimates~(\ref{NSE-est}) hold (because of~(\ref{d-b-w}) and homogeneous boundary condition~$\we|_{\partial\Omega}\equiv\0$\,). 
This proves the existence statement of Theorem~\ref{th:main} in Case (I). Therefore, from now on we will always assume that an alternative option is taking place,  that means,
\newline

\underline{Case (II):} there holds
$$\lim\limits_{k\to \infty} J_k = +\infty \, .$$
This case is much more involved because, it relies on the study of several properties of weak solutions to a stationary Euler system, see \eqref{Euler0} below.
By construction, for any $k \in \mathbb{N}$, the normalized functions in \eqref{pp6} satisfy (in strong form) the system
\begin{equation} \label{NS-Ek}
	\left\{
\begin{aligned}
	& - \dfrac{1}{J_k} \Delta \widehat{\mathbf{w}}_k + (\widehat{\mathbf{w}}_k \cdot \nabla) \widehat{\mathbf{w}}_k + \dfrac{1}{J_k}(\mathbf{U} \cdot \nabla) \widehat{\mathbf{w}}_k + \dfrac{1}{J_k}(\widehat{\mathbf{w}}_k \cdot \nabla) \mathbf{U} + \nabla \widehat{p}_k = \dfrac{1}{J^{2}_{k}} \, \mathbf{g} \ \ \mbox{ in } \ \ \mathcal{B}_k \, , \\[4pt]
	& \nabla \cdot \widehat{\mathbf{w}}_k = 0 \ \ \mbox{ in } \ \ \mathcal{B}_k \, , \\[4pt]
	&\widehat{\mathbf{w}}_k = \0 \ \ \mbox{ on } \ \ \partial \mathcal{B}_k \, ,
\end{aligned}
\right.
\end{equation}
its weak formulation
\begin{equation} \label{weaksolutionshnormalized}
\dfrac{1}{J_k} \int_{\mathcal{B}_k} \nabla \widehat{\mathbf{w}}_k \cdot \nabla \bbe + \int_{\mathcal{B}_k} (\widehat{\mathbf{w}}_k \cdot
\nabla)\widehat{\mathbf{w}}_k \cdot \bbe + \dfrac{1}{J_k} \int_{\mathcal{B}_k} [(\mathbf{U} \cdot \nabla) \widehat{\mathbf{w}}_k +
(\widehat{\mathbf{w}}_k \cdot \nabla) \mathbf{U}] \cdot \bbe = \dfrac{1}{J^{2}_{k}} \int_{\mathcal{B}_k} \mathbf{g} \cdot \bbe
\end{equation}
for all $\bbe \in W^{1,2}_{0, \sigma}(\mathcal{B}_k)$, together with the identity
\begin{equation} \label{hatu-1}
	\int_{  \mathcal{B}_{k}} | \nabla \widehat{\mathbf{w}}_k |^2 = \int_{\Omega} | \nabla \widehat{\mathbf{w}}_k |^2 =1 \, .
\end{equation}
Then, there exists a~vector field $\mathbf{v} \in W_{0,\sigma}^{1,2}(\Omega)$ such that, for every $t \in [2, +\infty)$,
\begin{equation} \label{wksconv2}
	\widehat{\mathbf{w}}_k \rightharpoonup \mathbf{v} \ \ \ \text{weakly in} \ W_{0,\sigma}^{1,2}(\Omega) \qquad \text{and} \qquad \widehat{\mathbf{w}}_k \to \mathbf{v} \ \ \ \text{strongly in} \ L_{\text{loc}}^{t}(\Omega) \ \ \ \text{as} \ \ k \to +\infty \, ,
\end{equation}
and, moreover,
\begin{equation} \label{310}
	\int_{\Omega} |\nabla\ve|^2\le1 \, .
\end{equation}
Proceeding as in \cite{kpr,korobkov2015solution}, taking the limit as $k \to +\infty$ in \eqref{weaksolutionshnormalized} allows us to obtain a scalar pressure $q \in L_{\text{loc}}^{2}(\Omega)$ such that
\begin{equation} \label{weaksolutionsh4}
\int_{\Omega} (\mathbf{v} \cdot \nabla)\mathbf{v} \cdot \bbe - \int_{\Omega} q (\nabla \cdot \bbe) = 0  \qquad \forall \bbe \in W^{1,2}_{0}(\Omega) \, .
\end{equation}
Therefore, the pair $(\mathbf{v},q) \in W_{0,\sigma}^{1,2}(\Omega) \times L_{\text{loc}}^{2}(\Omega)$ is a weak solution to the following stationary Euler system:
\begin{equation}  \label{Euler0}
	\left\{
\begin{aligned}
	& (\mathbf{v} \cdot \nabla) \mathbf{v}+ \nabla q = \0 \ \ \mbox{ in } \ \ \Omega \, ,\\[4pt] &\nabla \cdot \mathbf{v} = 0 \ \ \mbox{ in } \ \ \Omega \, , \\[4pt]
	& \mathbf{v} = \0 \ \ \mbox{ on } \ \ \partial \Omega \, .
\end{aligned}
\right.
\end{equation}
By Sobolev embedding we have $(\mathbf{v} \cdot \nabla) \mathbf{v} \in L^{s}_\loc(\bar\Omega)$ for every $s \in [1, 2)$, so that
\begin{equation} \label{311}
q \in W^{1,s}_{\text{loc}}(\bar\Omega) \quad \text{for every} \ s \in [1, 2) \, .
\end{equation}
The next assertion is one of the main novelties of the present paper and it is fundamental for our subsequent arguments. 

 \begin{theorem}\label{Euler=zero} Under above assumptions the solution $\ve$ to the Euler system \eqref{Euler0}, obtained through \eqref{wksconv2}, is
\textit{trivial}, that is,
\begin{equation}  \label{pp8}
	\ve = \0 \ \ \mbox{ almost everywhere in } \ \Omega \, .
\end{equation}
\end{theorem}
Since the proof of  \eqref{pp8} is quite lengthy, involving several subtle tools from real analysis and topology, we postpone it to Appendix~\ref{poet2}.  Note that, unlike the situation in \cite{korobkov2015solution}, the triviality of the limit solution {\bf does not} entail a contradiction and by no means ends the proof of Theorem~\ref{th:main}. However, identity~(\ref{pp8}) plays a key role in proving the Leray--Hopf inequality (see~(\ref{gg10}) below) and in establishing a~global linear estimate~(\ref{gg1}) 
for the~energy of the~constructed solutions.

\subsection{The global linear growth of the Dirichlet integral}
\label{poet4}
Let $\we_k$ be as in \eqref{NSE-wk},
$\mathcal{B}_k$ in Figure~\ref{fig:2}, and $\Omega^a$ in \eqref{Omega_ab}.

For the invading domains, from now on we assume that
\begin{equation}\label{ug7}
\Omega^{t_k}\subset  \mathcal{B}_k\subset \Omega^{t_k+5}\end{equation}
for some sequence of $t_k \to+\infty$. The aim of this section is to prove the following global estimate:
\begin{equation}\label{gg1}
J^2_k=\int\limits_{  \mathcal{B}_k}|\nabla\we_k|^2\le c_0 \,t_k
\end{equation}
for some $c_0>0$ independent of $k$. By contradiction, suppose that \eqref{gg1} is not true; then, up to a subsequence, we have
\begin{equation}\label{gg3-0}
\frac{J_k^2}{t_k}\to+\infty\quad\mbox{ as }\ k \to +\infty.
\end{equation}

Taking (\ref{weaksolutionsh}) with $\boldsymbol{\eta}\doteq\we_k$, we get
\begin{align}
	J_k^2&= \int\limits_{  \mathcal{B}_k}\mathbf{g}\cdot\we_k-\int\limits_{  \mathcal{B}_k}(\mathbf{w}_k \cdot \nabla) \mathbf{U}\cdot\we_k  \nonumber \\[4pt]
	&=\int\limits_{  \mathcal{B}_k}\mathbf{g}\cdot\we_k-\int\limits_{  \mathcal{B}_k\setminus\Omega^2}(\mathbf{w}_k \cdot \nabla) \mathbf{U}\cdot\we_k-\int\limits_{\Omega^2}(\mathbf{w}_k \cdot \nabla) \mathbf{U}\cdot\we_k=: I_1+I_2+I_3.\label{gg3}
\end{align}
Due to (\ref{uef}) and (\ref{uboundz2}), we have $\|\mathbf{g}\|_{L^2(  \mathcal{B}_k)}\lesssim \sqrt{t_k}$ and, hence,
by the H\"older inequality and \eqref{gg3-0} (combined with the Poincaré inequality), we obtain
\begin{equation}\label{gg4}
|I_1|=\e_k\, J_k^2
\end{equation}
with $\e_k\to0$ as $k\to\infty$. Further, by the Leray-Hopf inequality (\ref{HI}), we have
\begin{equation}\label{gg5}
|I_2|\le\frac18\,{J_k^2}.
\end{equation}
Therefore, (\ref{gg3}) implies that
$$
{J_k^2}\le 2 I_3=-2\int\limits_{\Omega^2}(\mathbf{w}_k \cdot \nabla) \mathbf{U}\cdot\we_k
$$
for $k$ large enough. Let $\widehat\we_k$ be as in \eqref{pp6}, then the last inequality reads
\begin{equation}\label{gg7}-2\int\limits_{\Omega^2}(\widehat\we_k \cdot \nabla) \mathbf{U}\cdot\widehat\we_k\ge1
\end{equation}
for $k$ large enough.
Since $\Omega^2$ is a~bounded domain with Lipschitz boundary, we have the~weak convergence $\widehat\we_k\rightharpoonup\ve$ in $W^{1,2}(\Omega^2)$ and the~strong convergence
$\widehat\we_k\to\ve$ in any $L^s(\Omega^2)$ with $s<\infty$, where $\ve$ is the corresponding limiting solution to the Euler system~(\ref{Euler0}). Then, letting $k\to\infty$ in (\ref{gg7}), we get
\begin{equation}\label{gg8}-2\int\limits_{\Omega^2}(\ve \cdot \nabla) \mathbf{U}\cdot\ve\ge 1.
\end{equation}
This inequality contradicts (\ref{pp8}), thereby proving the~global linear growth estimate~(\ref{gg1}).

\subsection{The uniform linear growth of the Dirichlet integral}\label{poet6}

Let $\we_k$ be as in \eqref{NSE-wk} and $\Omega^t$ be as in \eqref{Omega_ab}.
If the following double limit is finite
$$
\varliminf_{t\to +\infty} \, \varliminf_{k\to +\infty} \int\limits_{\Omega^t}|\nabla\we_{k}|^2 < +\infty,
$$
then it is straightforward to obtain \eqref{NSE-est} and to finish the proof of Theorem~\ref{th:main} (as in Case~I in \linebreak Subsection~\ref{proofs-sec}\,). So, from now on, we assume that
\begin{equation}\label{unbw}
	\lim_{t\to +\infty} \, \varliminf_{k\to +\infty} \int\limits_{\Omega^t}|\nabla\we_{k}|^2 = +\infty.
\end{equation}
The main tool for the subsequent arguments is the following Leray--Hopf type inequality. 

\begin{theorem}\label{prep-l2}
Let $t_k$ be as in \eqref{ug7}. If \eqref{unbw} holds, then there exists $t_*>3$ and $k_*\in\N$ such that
	for all $k\ge k_*$, we have
	\begin{equation}\label{gg11}\biggl|\int\limits_{\Omega^2}(\mathbf{w}_k \cdot \nabla) \mathbf{U}\cdot\we_k\biggr|\le
		\frac18\int\limits_{\Omega^{t_*}}|\nabla\we_k|^2, 
	\end{equation}
	\begin{equation}\label{gg10}\biggl|\int\limits_{\Omega^t}(\mathbf{w}_k \cdot \nabla) \mathbf{U}\cdot\we_k\biggr|\le
		\frac14\int\limits_{\Omega^t}|\nabla\we_k|^2\qquad \forall t\in [t_*, t_k].
	\end{equation}
\end{theorem}

\begin{proof} 
	We prove \eqref{gg11} by contradiction. Assume it fails, then, by passing to a subsequence,  there exists a strictly increasing sequence of numbers
	$t^*_k \to +\infty$  such that $t^*_k<t_k$ and 
	\begin{equation}\label{gg10-2}
		\biggl|\int\limits_{\Omega^2}(\mathbf{w}_{k} \cdot \nabla) \mathbf{U}\cdot\we_{k}\biggr|>
		\frac18\int\limits_{\Omega^{t^*_k} }|\nabla\we_{k}|^2.
	\end{equation}
Denote
	\begin{equation}
		\widetilde{J}_k \doteq \left(\int_{\Omega^{t^*_k}} |\nabla \we_k|^2\right)^\frac12,\qquad\widetilde{\we}_k\doteq\frac1{\widetilde{J}_k}\we_k, \qquad \widetilde{p}_k\doteq\frac1{ \bigl(\widetilde{J}_k\bigr)^2 }p_k.
	\end{equation}
	By \eqref{unbw}, we have $\widetilde{J}_k \to +\infty$. Hence, $\widetilde{\we}_k$ converges weakly to an Euler solution $\widetilde{\ve}$
in $\Omega$ and, moreover, as for (\ref{Euler0})--\eqref{pp8}, we have the identity
	\begin{equation} \label{tildeve-0}
		\widetilde{\ve} \equiv 0.
	\end{equation}
	From \eqref{gg10-2}, we deduce that
	\begin{equation}
		\biggl|\int\limits_{\Omega^2}(\widetilde{\mathbf{w}}_{k} \cdot \nabla) \mathbf{U}\cdot\widetilde{\we}_{k}\biggr|>
		\frac18,
	\end{equation}
	and taking the limit for $k \to +\infty$,  we get
	\begin{equation} \label{519}
		\biggl|\int\limits_{\Omega^2}(\widetilde{\ve} \cdot \nabla) \mathbf{U}\cdot\widetilde{\ve}\biggr|\ge\frac18,	\end{equation}
which contradicts~\eqref{tildeve-0} and proves~\eqref{gg11}.  Using (\ref{HI}) and \eqref{gg11}, we infer the second bound~\eqref{gg10}.	
\end{proof}

The main result of this section is a uniform growth estimate for the Dirichlet integrals, see Lemmas~\ref{lem:61}--\ref{lem:62} below.
We prove it by using the Ladyzhenskaya \& Solonnikov approach~\cite{ladyzhenskaya1983determination}, with adaptations to our more general
setting that requires Theorem~\ref{prep-l2} instead of the usual Leray-Hopf inequality. For the reader's convenience, we will repeat
the main ideas of \cite{ladyzhenskaya1983determination} in detail but with some simplifications. We need the following slight variant of an elementary comparison principle from one-dimensional real analysis due to Ladyzhenskaya \& Solonnikov \cite[Lemma 2.3]{ladyzhenskaya1983determination}.

\begin{lemma}\label{cc-l1} {\sl Let $h(t)$ and $\varphi(t)$ be smooth nondecreasing nonnegative functions satisfying the following estimates:
\begin{equation}\label{ug2}h(t)\le\Psi\bigl(h'(t)\bigr)+\frac12\varphi(t)\qquad\forall t\in[t_0,T],
\end{equation}
\begin{equation}\label{ug3}\vp(t)\ge 2\,\Psi\bigl(\vp'(t)\bigr)\qquad\forall t\in[t_0,T],\end{equation}
where $\Psi:\R_+\to\R_+$ is a~nondecreasing function. Suppose also that
\begin{equation}\label{ug4}h(T)\le \vp(T),\end{equation}
then
\begin{equation}\label{ug5}h(t)\le \vp(t)\qquad
\forall t\in [t_0,T].\end{equation}}
\end{lemma}

\medskip

\begin{proof}
 	 Suppose that
	\begin{equation}\label{ug6}h(t)>\vp(t)\end{equation}
	for some $t\in [t_0,T)$, then
	$$\Psi\bigl(h'(t)\bigr)\overset{(\ref{ug2})}\ge h(t)-\frac12\vp(t)\overset{(\ref{ug6})}>\frac12\vp(t)\overset{(\ref{ug3})}\ge\Psi\bigl(\vp'(t)\bigr),$$
	therefore,
	$$h'(t)>\vp'(t)$$
	for all $t$ satisfying~(\ref{ug6}). Hence, (\ref{ug6}) remains true until $t=T$, which contradicts \eqref{ug4}.
\end{proof}

\begin{lemma} \label{lem:61} Under the assumption \eqref{unbw}, there holds that
\begin{equation}\label{ug1}
D_k(t) \doteq\int\limits_{\Omega^t }|\nabla\we_k|^2\le c_0t+c_1, \qquad \forall t\ge0,
\end{equation}
for some $c_0, c_1>0$ independent of~$k$ and $t$.
\end{lemma}
\begin{proof} 
We assume without loss of generality that, for all $k\in\N$, the inclusion $\Omega^{t_*} \subset  \mathcal{B}_k$ holds and (\ref{gg11}) is fulfilled. In view of \eqref{gg1}, it is sufficient to prove \eqref{ug1}
for $t \le t_k$, where $t_k$ is from \eqref{ug7} (recall, that $\we_k$ is extended by zero outside of~$\mathcal{B}_k$). Testing (\ref{NSE-wk}) with $\we_k$ in
$\Omega^{t}$ and by means of an integration by parts (recall $\we_k|_{\partial  \mathcal{B}_k}\equiv\0$), we deduce that
	\begin{equation}\label{ug9}D_k(t)=\int\limits_{\OO^t}\mathbf{g}\cdot\we_k-\int\limits_{\OO^t}(\mathbf{w}_k \cdot \nabla) \mathbf{U}\cdot\we_k+\frac12\int\limits_{\sigma^t}\biggl( \nabla |\we_k|^2-|\we_k|^2\ue_k-2\,p_k\we_k \biggr)\cdot\n,\qquad \forall  t \in [3, t_k],
	\end{equation}
	where
	\begin{equation}\label{ug10}\sigma^t=\bigcup\limits_{j=1}^J\sigma^t_j.
	\end{equation}
	Applying (\ref{gg10}), we obtain
	\begin{equation}\label{ug11}
		\frac34D_k(t)\le\int\limits_{\OO^t}\mathbf{g}\cdot\we_k+\frac12\int\limits_{\sigma^t}\biggl(\nabla |\we_k|^2-|\we_k|^2\ue_k-2\,p_k\we_k\biggr)\cdot\n,\qquad \forall t \in [t_*, t_k].
	\end{equation}
	By the geometry of $\Omega^t$ and Poincar\'e's inequality, there holds
	\begin{equation}\label{ug-est-s1}\int\limits_{\OO^t}|\we_k|^2\le \cc\, D_k(t),\qquad\forall t\ge 0,\end{equation}
	where $\cc>0$ is some constant (independent of~$k$ and $t$). By H\"older's inequality and \eqref{ug11}, we get
	\begin{equation}\label{ug-est-s2}
		\int\limits_{\OO^t}\mathbf{g}\cdot\we_k\le \biggl(\tilde c\,D_k(t)\,G(t)\biggr)^\frac12\le \frac38 D_k(t)+\frac23\cc\,G(t),\end{equation}
	where $G(t)\doteq\int\limits_{\OO^t}|\mathbf{g}|^2$.
	Therefore, from \eqref{ug11} and \eqref{ug-est-s2} we deduce that
	\begin{equation}\label{ug12}D_k(t)\le\frac{16 \cc}9 \,G(t)+\frac43\int\limits_{\sigma^t}\biggl[\nabla |\we_k|^2-|\we_k|^2\ue_k-2\,p_k\we_k\biggr]\cdot\n, \qquad \forall t \in [t_*, t_k].
	\end{equation}
	If we put
\begin{equation}\label{ug13}
h_k(t)\doteq \int\limits_{t-1}^t D_k(\tau)\,d\tau,\qquad\Omega^{t-1,t}=\Omega^t\setminus\overline{\Omega^{t-1}},
\end{equation}
then by integrating (\ref{ug12}) over $[t-1, t]$, we find
	\begin{equation}\label{ug15}
		h_k(t)\le\frac{16 \cc}9 \,G(t)+\frac43\int\limits_{\Omega^{t-1,t}}\biggl(\nabla |\we_k|^2-|\we_k|^2\ue_k-2\,p_k\we_k\biggr)\cdot\n, \qquad \forall t \in [t_*+1, t_k].
	\end{equation}
	By H\"older's inequality, \eqref{ug-est-s1} and standard estimates for the pressure,
	\begin{equation}\label{ug16}
		\biggl|\int\limits_{\Omega^{t-1,t}}\biggl(\nabla |\we_k|^2-|\we_k|^2\ue_k-2\,p_k\we_k\biggr)\cdot\n\biggr|\le c_* \biggl[\  \int\limits_{\Omega^{t-1,t}}|\nabla\we_k|^2+\biggl(\int\limits_{\Omega^{t-1,t}}|\nabla\we_k|^2\biggr)^\frac32\biggr]
	\end{equation}
	with some constant~$c_*>0$ (see, e.g., \cite[proof of Theorem~2.2]{ladyzhenskaya1983determination} for details). We emphasize that the zero flux of $\we_k$ through the channels (i.e., $\int_{\sigma_j^t} \we_k \cdot \mathbf{n} \, ds = 0$)  plays a key role in the proof of the last inequality: indeed, it implies that the left hand side of~(\ref{ug16}) does not change if we add any constant to the pressure, in other words, the left hand side depends on  the pressure only through its derivatives.
	
Next, notice that (\ref{ug13}) implies the identity
	\begin{equation}\label{ug19}h_k'(t)=D_k(t)-D_k(t-1)=\int\limits_{\Omega^{t-1,t}}|\nabla\we_k|^2.
	\end{equation}
	Hence, from (\ref{ug15})--(\ref{ug19}) we get the differential inequality
	\begin{equation}\label{ug20}
		h_k(t)\le\frac{16 \cc}9 \,G(t)+\frac43c_*\,\biggl[h_k'(t)+h_k'(t)^\frac32\biggr]\le c_2\biggl[G(t)+h_k'(t)+h_k'(t)^\frac32\biggr],\qquad \forall t \in [t_* + 1, t_k],\end{equation}
	where we denote
	$c_2=\max\bigl\{\frac{16}9 \cc, \frac43c_*\bigr\}$.
Define $\varphi(t)\doteq c_3\,t +c_4$, where the positive constants $c_3$ and $c_4$ are chosen such that
\begin{equation}\label{ug21}
h_k(t_k)\le\vp(t_k),\qquad 2c_2 \bigl(c_3+c_3^\frac32\bigr)\le\vp(t_*+1),\qquad 2c_2\,G(t)\le \vp(t),\qquad\forall t\in[t_*+1,t_k].\end{equation}
Note that the first requirement in (\ref{ug21}) can be achieved by the global bound~(\ref{gg1}) and the inequality $h_k(t)\le D_k(t)$; the second requirement in (\ref{ug21}) can be achieved if we enlarge $c_4$ for the fixed~$c_3$, and the third requirement in (\ref{ug21}) can be achieved by the linear growth of $G(t)$. Finally, we put
	$$\Psi(\tau)\doteq c_2\bigl(\tau+\tau^\frac32\bigr).$$
	By construction, all the assumptions \eqref{ug2}--\eqref{ug4}  are fulfilled with $T=t_k$. Therefore, by~(\ref{ug5}),
	$$D_k(t-1)\le h_k(t)\le c_3\,t+c_4,\qquad\forall t\in[t_*+1, t_k].$$
The last assertion implies the desired estimate~(\ref{ug1}) with suitably chosen constants~$c_0$ and $c_1$.\end{proof}

Next, using the same type of argument we obtain a~similar 'backward' growth estimate.

\begin{lemma} \label{lem:62}
	Under the assumption \eqref{unbw}, there holds that
	\begin{equation}\label{ug1--}
		\int\limits_{ \mathcal{B}_k \setminus \Omega^t}|\nabla\we_k|^2\le c_0 (t_k-t)+c_1, \qquad \forall t \in [0, t_k],
	\end{equation}
for some $c_0, c_1>0$ independent of~$k$ and $t$.
\end{lemma}

\begin{proof}
	Testing the system~(\ref{NSE-wk}) by $\we_k$ in the domain $ \mathcal{B}_k \setminus \Omega^t$ and using standard integration by parts, we get
	\begin{equation}\label{ug9-}
		\int\limits_{ \mathcal{B}_k \setminus \Omega^t}|\nabla\we_k|^2=\int\limits_{ \mathcal{B}_k \setminus \Omega^t}\mathbf{g}\cdot\we_k-\int\limits_{ \mathcal{B}_k \setminus \Omega^t}(\mathbf{w}_k \cdot \nabla) \mathbf{U}\cdot\we_k+\frac12\int\limits_{\sigma^t}\biggl(\nabla |\we_k|^2-|\we_k|^2\ue_k-2\,p_k\we_k\biggr)\cdot\n,
	\end{equation}
	where $\mathbf{n}$ is now the outward normal vector with respect the domain $ \mathcal{B}_k \setminus \Omega^t$. Repeating the arguments from Lemma \ref{lem:61} (with $t_k-t$ playing the role of $t$ there)  and using \eqref{HI} instead of Theorem~\ref{prep-l2}, we get~(\ref{ug1--}).
\end{proof}


\subsection{The  uniform local estimates  of the Dirichlet integral}
\label{poet7}

The aim of this section is to prove the following uniform estimate.

\begin{lemma}
Under the assumption \eqref{unbw} there exists $c_5>0$, independent of~$\tau$ and $k$, such that
\begin{equation}\label{lg1}
\int\limits_{\Omega^{\tau-1, \tau}}|\nabla\we_k|^2\le c_5, \qquad \forall \tau \in [1,t_k],
\end{equation}
\end{lemma}
\begin{proof}
If $\tau$ is near $1$ (say, $\tau \le10$), then \eqref{lg1} follows from (\ref{ug1}). If $\tau$ is near $t_k$ (say, $\tau\ge t_k-10$), then
\eqref{lg1} follows from~(\ref{ug1--}).\par
The delicate situation is when $10<\tau<t_k-10$. In this case, we test (\ref{NSE-wk}) by $\we_k$ in $\Omega^{\tau-1,\tau}$  to get
\begin{equation}\label{lg3}		\int\limits_{\Omega^{\tau-1,\tau}}|\nabla\we_k|^2=\int\limits_{\Omega^{\tau-1,\tau}}\mathbf{g}\cdot\we_k-\int\limits_{\Omega^{\tau-1,\tau}}(\mathbf{w}_k \cdot \nabla) \mathbf{U}\cdot\we_k+I_\tau-I_{\tau-1},
\end{equation}
where
$$
I_{\tau}\doteq\frac12\int\limits_{\sigma^\tau}\biggl(\nabla |\we_k|^2-|\we_k|^2\ue_k-2\,p_k\we_k\biggr)\cdot\n,
$$
with $\n$ being the outward normal vector with respect to~$\Omega^\tau$. We now distinguish three possible cases.\par
{\sc Case I}: $I_{\tau-1}\ge0$. In this case, we put
$$\WO(t) \doteq\Omega^{t+\tau-1}\setminus{\bar\Omega}^{\tau-1}=\Omega^{\tau-1,t+\tau-1}\qquad \forall
t\in [0, t_k-\tau+1],\qquad\WD_k(t)\doteq\int\limits_{\WO(t)}|\nabla\we_k|^2.$$
By an energy estimate similar to \eqref{lg3}, we find
\begin{equation}\label{lg5}
\WD_k(t)=\int\limits_{\WO(t)}\mathbf{g}\cdot\we_k-\int\limits_{\WO(t)}(\mathbf{w}_k \cdot \nabla) \mathbf{U}\cdot\we_k+I_{t+\tau-1}-I_{\tau-1}\le \int\limits_{\WO(t)}\mathbf{g}\cdot\we_k-\int\limits_{\WO(t)}(\mathbf{w}_k \cdot \nabla) \mathbf{U}\cdot\we_k+I_{t+\tau-1}\quad\forall t\in[0,T]
\end{equation}
where $T\doteq t_k-\tau+1$. By (\ref{ug1--}) we then find $\WD_k(T)\le c_0 T+c_1$ and,
repeating the arguments for Lemma~\ref{lem:61}, we get
$$\WD_k(t)\le c_0 t +c_1,\qquad\forall t\in [0, T]$$
and the required estimate~(\ref{lg1}) follows by taking $t=1$ therein.\par
{\sc Case II}: $I_{\tau}\le0$. In this case, we put
$$\WO(t)=\Omega^{\tau}\setminus\bar\Omega^{\tau-t}=\Omega^{\tau-t,\tau}\qquad\forall t\in [0, \tau],
\qquad\WD_k(t)=\int\limits_{\WO(t)}|\nabla\we_k|^2.$$
	By an energy estimate similar to \eqref{lg3}, we have
	\begin{equation}\label{lg5}\WD_k(t)\le \int\limits_{\WO(t)}\mathbf{g}\cdot\we_k-\int\limits_{\WO(t)}(\mathbf{w}_k \cdot \nabla) \mathbf{U}\cdot\we_k-I_{\tau-t}\qquad\forall t\in [0,T]
\end{equation}
where, now, $T=\tau$. By (\ref{ug1}) we have $\WD_k(T)\le c_0 T+c_1$ and, repeating again the arguments for Lemma~\ref{lem:61},
we get
$$\WD_k(t)\le c_0 t +c_1\qquad\forall  t\in [0, T]$$
and the required estimate~(\ref{lg1}) follows by taking $t=1$ therein.\par
	{\sc Case III}: $I_{\tau}>0$ and $I_{\tau-1}<0$. Then  $I_{\tau_*}=0$ for some $\tau_*\in (\tau-1,\tau)$. This case can be covered by a~combination  of the  previous two cases. Consider domains $\WO_+(t)$ and $\WO_-(t)$ of the forms
$$\WO_+(t)=\Omega^{t+\tau_*}\setminus\Omega^{\tau_*}\quad\forall t\in [0, t_k-\tau_*],\qquad
\WO_-(t)=\Omega^{\tau_*}\setminus\Omega^{\tau_*-t}\quad\forall t\in [0, \tau_*].$$
	Similar to Cases I and II, we have the uniform estimates within both $\WO_+(t)$ and $\WO_-(t)$:
	\begin{equation}
		\int_{\WO_-(t)} |\nabla \we_k|^2\le c_0 t+ c_1, \qquad \ \ \int_{\WO_+(t)} |\nabla \we_k|^2 \le c_0 t+ c_1,
	\end{equation}
	for $t$ in the corresponding ranges. Then (\ref{lg1}) follows by summing of these inequalities with~$t= 1$.\par	
Combining the above three cases, the inequality~(\ref{lg1}) is proved completely.\end{proof}

\subsection{Conclusion of the proof of Theorem \ref{th:main}} \label{conclusionproof}

As already mentioned in Subsection \ref{poet6}, without loss of generality we can assume \eqref{unbw}. By the estimates~(\ref{ug1}), we can extract a~subsequence of $\we_k$ which converges weakly to some~$\we\in W^{1,2}_\loc(\overline\Omega)$. Then, $\we$ is a weak solution to~(\ref{NSE-w}).  By the estimates~(\ref{ug1}), (\ref{lg1}), we have
\begin{equation}  \label{NSE-les}
	\|\nabla\we\|_{L^2(\Omega\cap B(z_0))}\le C_*.
\end{equation}
for any unit disk $B(z_0)$ centered at any point $z_0\in\R^2$, where $C_*$ is some constant independent of~$z_0$.
Then, from the homogeneous boundary conditions~$\we|_{\partial\Omega}\equiv\0$ and from the regularity theory for the Stokes system (see, e.g., \cite[Section~IV.6]{galdi2011introduction} or \cite[Section~2.4.2]{korobkov2024steady}),
we infer the required uniform estimates \eqref{NSE-est} for $\ue=\we+\mathbf{U}$ and for the corresponding pressure~$p$.

\section{Proof of Theorem \ref{th:main2}}
\label{poet8}

Let the assumptions of Theorem~\ref{th:main2} be fulfilled with some small parameter~$\delta_*>0$ whose smallness will be specified below.
Let $\ue$ be a solution from Theorem~\ref{th:main}, i.e., $\ue$ is a~solution to the system~(\ref{NSE}) satisfying the estimate~(\ref{NSE-est}).
Further we will use some constructions and identities from the proof of Theorem~\ref{th:main} with the following modifications (simplifications).
First of all, in the construction of the flux carrier in Subsection~\ref{sec:2.1} and in Appendix~\ref{appendix}
we now simply set $\UU = \mathbf{CP}_j$ in each outlet $\Omega_j^{2, \infty}$ (inside $\Omega_0$ we still take $\UU = \mathbf{V}$ with $\mathbf{V}$ solving~(\ref{ttO})\,):
indeed, if $\delta_*$ is small enough, then we have  for any $j \in \{1,2,\cdots, J\}$ and any $2 \le a<b$,
\begin{equation}\label{HI-222}
	\left| \int_{ \Omega_j^{a,b}} (\bbe\cdot \nabla)\mathbf{U} \cdot \bbe \right|  \leq \varepsilon \| \nabla\bbe \|_{L^{2}( \Omega_j^{a,b})}^2,
\end{equation}
if $\bbe$ vanishes on $\bigl(\partial \Omega\bigr) \cap \partial \Omega_j^{a,b}$.
Moreover, we now have
$\Delta \mathbf{U} - (\mathbf{U} \cdot \nabla) \mathbf{U}\equiv\const$
on every outlet.
Hence, we deduce that $\we\doteq\ue-\UU$ is a solution to the system~(\ref{NSE-w}) with the right-hand side given by
\begin{equation}  \label{uni-8}
	\mathbf{g}(z)\equiv \mathbf{f}(z)\qquad\forall z\in\Omega\setminus\Omega^2
\end{equation}
because the corresponding constant terms in channels can be absorbed into the gradient of the pressure.
By contradiction, suppose that the claim of Theorem~\ref{th:main2} is false, that is,
$$
\int\limits_{\Omega}|\nabla\we|^2= +\infty.
$$
Then, if we put $D(t)\doteq\int\limits_{\Omega^t}|\nabla\we|^2$, we have
\begin{equation}  \label{uni-10}
	D(t)\to+\infty\qquad\mbox{ as \ }t\to+\infty.
\end{equation}
Take $t_0>0$ such that
\begin{equation}\label{uni-gg11}\biggl|\int\limits_{\Omega^2}(\mathbf{w} \cdot \nabla) \mathbf{U}\cdot\we\biggr|\le
\frac18\int\limits_{\Omega^{2,t_0-1}}|\nabla\we|^2,\qquad 4 \cc\, G_\infty\le D(t_0-1),
\end{equation}
where $\cc$ is the~constant from~(\ref{ug-est-s1}) and $G_\infty\doteq\int\limits_\Omega|\mathbf{g}|^2$.
From (\ref{uni-gg11}) and (\ref{HI}) we get
\begin{equation}\label{uni-gg13}\biggl|\int\limits_{\Omega^t}(\mathbf{w} \cdot \nabla) \mathbf{U}\cdot\we\biggr|\le
	\frac14\int\limits_{\Omega^{t}}|\nabla\we|^2\qquad \forall t\ge t_0.
\end{equation}
Then, deducing formula~(\ref{ug12}) for $\we$, we have
\begin{equation}\label{uni-ug12}D(t)\le\frac{16}9 \cc\,G(t)+\frac43\int\limits_{\sigma^t} \left(\nabla |\we|^2- |\we|^2\ue-2p\we \right)
\cdot\n\overset{(\ref{uni-gg11})}<
	\frac12 D(t)+\frac43\int\limits_{\sigma^t}\left(\nabla |\we|^2- |\we|^2\ue-2p\we \right)\cdot\n,\end{equation}
therefore,
\begin{equation}\label{uni-ug13}D(t)\le\frac83\int\limits_{\sigma^t}\left(\nabla |\we|^2- |\we|^2\ue-2p\we \right)\cdot\n\qquad \forall t\ge t_0.\end{equation}
Then, formula~(\ref{ug20}) takes the form
\begin{equation}\label{uni-ug20first}
h(t)\le c_2\biggl[h'(t)+h'(t)^\frac32\biggr]\qquad \forall t\ge t_0,
\end{equation}
where, as before, $h(t)\doteq
\int\limits_{t-1}^t D(\tau)\,d\tau$ and $c_2>0$ is some constant.
But \eqref{uni-10} and \eqref{uni-ug20first} imply $h'(t)\to+\infty$ as $t\to+\infty$, and this contradicts the~linear estimate~(\ref{NSE-est}). Hence,
$$
\int_{\Omega}|\nabla\we|^2 < +\infty.
$$
meaning that \eqref{uni-2} holds. From the estimates for the solutions to the Stokes system (see, e.g., \cite[Section~IV.6]{galdi2011introduction} or
\cite[Section~2.4.2]{korobkov2024steady}\,) we obtain also that
$$
\lim\limits_{t \to +\infty} \| \ue-\textbf{CP}_{j} \|_{W^{2,2}(\Omega^{t,t+1}_{j})} = 0 \qquad \forall j \in \{1,\dots,J \} \, ,
$$
which, by Sobolev embedding, implies the uniform convergence \eqref{convergencecp}. This concludes the proof.
\hfill$\qed$

\section{Proof of Theorem \ref{th:main3}}
\label{poet8.2}

We are going to deduce Theorem~\ref{th:main3} as a direct consequence of the following two claims:

\begin{proposition}[{\bf Smallness}]\label{th:main4} Let the assumptions of Theorem~\ref{th:main3} be fulfilled. Then for any $\e>0$
there exists $\delta_{\e}=\delta_{\e}(\Omega)>0$ such that if
		\begin{equation}  \label{uni-21}
			\|\fe\|_{L^2(\Omega)}+\|\ae\|_{W^{3/2,2}(\partial\Omega \cap B_{R})}+\max\limits_{j=1,\dots,J}\left( |\be_j^0|
+|\be_j^1|+ |\Fe_j| \right)<\delta_\e ,
		\end{equation}
		then there exists a solution~$\ue$ from Theorem~\ref{th:main} satisfying \eqref{NSE-est} with the parameter $C_*<\e$ and
		\begin{equation}  \label{uni-22}
			\sup\limits_{z\in\Omega}|\ue(z)|<\e.
		\end{equation}
\end{proposition}

\begin{proposition}[{\bf Weak-Strong Uniqueness}]\label{th:main5} Let the assumptions of Theorem~\ref{th:main3} be fulfilled. Then there exists~$\e_0=\e_0(\Omega)>0$ such that if $\ue\in W^{1,2}_\loc(\overline\Omega)$ is a solution to~\eqref{NSE} satisfying
		\begin{equation}  \label{uni-23}
			\esssup\limits_{z\in\Omega}|\ue(z)|<\e_0,\end{equation}
		then for any other
		solution~$\tilde\ue\in W^{1,2}_\loc(\overline{\Omega})$ to the same problem~\eqref{NSE} the estimate from below
		\begin{equation}  \label{uni-25}
			\varliminf\limits_{t\to+\infty}\frac{D(t)}{t^3}>0
		\end{equation}
		holds, where we denoted 
		\begin{equation}  \label{uni-26}
			D(t)\doteq\int\limits_{\Omega^t}|\nabla\we|^2,\qquad \we=\tilde\ue-\ue.
		\end{equation}
\end{proposition}

\medskip
{\sc Proof of Proposition~\ref{th:main4}}. Fix $\e\in(0,1)$. Let \eqref{uni-21} be fulfilled for some~$\delta_\e>0$, whose smallness will be specified below. Similar to the proof of Theorem \ref{th:main2}, in the construction
of the flux carrier in Subsection~\ref{sec:2.1} we can put $\mathbf{U} = \mathbf{CP}_j$ in the outlets. Fix $N\in\mathbb{N}$ (whose value will be specified below). From the estimates for the solutions to the Stokes
system (see, e.g., \cite[Section~IV.6]{galdi2011introduction} or
\cite[Section~2.4.2]{korobkov2024steady}\,)  we obtain, that for $\delta_\e$ small enough the function $\fe$ and the new flux carrier $\UU$ satisfy the inequalities
\begin{equation}  \label{uni-27}
\|\fe\|_{L^2(\Omega)}<\frac{\e}{N},\qquad\esssup\limits_{z\in\Omega}|\UU(z)|<\frac{\e}{N},\qquad\|\UU\|_{W^{2,2}(\Omega\cap B(z_0))}<\frac{\e}{N}\qquad\forall z_0\in\R^2.
\end{equation}
Let $\ue_k$ be a solution to~(\ref{NSE-k}) and let $\we_k\doteq\ue_k-\UU$ be a solution to~(\ref{NSE-wk}) with
\begin{equation}  \label{uni-29}
\mathbf{g}(z)\equiv \mathbf{f}(z)\qquad\forall z\in\Omega\setminus\Omega^2; \qquad\qquad \mathbf{g}\equiv\mathbf{f}+\Delta \mathbf{U} - (\mathbf{U} \cdot \nabla) \mathbf{U}\qquad\forall z\in\Omega^2,
\end{equation}
(again, $\Delta \mathbf{U} - (\mathbf{U} \cdot \nabla) \mathbf{U}\equiv\const$ in every channel $\Omega_j\setminus\Omega^2$, and these constant
terms may be absorbed into $\nabla p$). Recalling \eqref{pp6}, from (\ref{gg3}) we infer
\begin{equation}\label{uni-32}
J_k^2= \int\limits_{\mathcal{B}_k}\mathbf{g}\cdot\we_k+\int\limits_{\mathcal{B}_k}(\mathbf{w}_k \cdot \nabla) \we_k\cdot\mathbf{U}\doteq I_1+I_2.
\end{equation}
By \eqref{uni-27} we deduce
\begin{equation}\label{uni-33-1}
\exists c=c(\Omega)>0\, :\quad\|\mathbf{g}\|_{L^2(\Omega)}<\frac{c\, \e}{N}\,,
\end{equation}
therefore, 
\begin{equation}\label{uni-33} I_1\le \frac14 J_k^2+\frac{\cc_0}{N^2}\e^2 \quad\mbox{ for some \ }\cc_0=\cc_0(\Omega)>0\, .
\end{equation}
Similarly, by (\ref{uni-27}),
\begin{equation}\label{uni-34}
\exists \cc_1=\cc_1(\Omega)>0\, :\quad I_2\le \frac{\cc_1 \varepsilon}{N} \,J_k^2\, .
\end{equation}
Hence, for $N>4\cc_1$ we get $J_k^2\le2{\cc_0}\e^2/{N^2} $ and, consequently, for the limiting solution $\we$ we obtain
$$\int\limits_\Omega|\nabla\we|^2\le2\frac{\cc_0}{N^2}\e^2.$$
Then 
from the homogeneous boundary conditions $\we|_{\partial\Omega}\equiv\0$   and  from the estimates for the solutions to the Stokes system (see, e.g., \cite[Section~IV.6]{galdi2011introduction} or
\cite[Section~2.4.2]{korobkov2024steady}\,) we get
\begin{equation}  \label{uni-35}
\exists \cc_2=\cc_2(\Omega)>0\, :\quad\|\we\|_{W^{2,\frac32}(\Omega\cap B(z_0))}<\cc_2\frac{\e}{N},\qquad\forall z_0\in\R^2.
\end{equation}
Therefore, by Sobolev Embedding Theorem, 
\begin{equation}  \label{uni-35---}
\exists \cc_3=\cc_3(\Omega)>0\, :\quad\|\we\|_{L^\infty(\Omega)}<\cc_3\frac{\e}{N}.
\end{equation}
Now the estimate~(\ref{uni-22}) follows from~(\ref{uni-35---}) and (\ref{uni-27})$_{2}$ if we take  $N>\max\{2,2\cc_3\}$.  \hfill $\qed$

\medskip
{\sc Proof of Proposition~\ref{th:main5} and of Theorem \ref{th:main3}}. Let $\ue\in W^{1,2}_\loc(\overline\Omega)$ be a~weak solution
to~(\ref{NSE}) satisfying~(\ref{uni-23}), and let $\tu\doteq \ue+\we\in W^{1,2}_\loc(\overline\Omega)$ be another solution to the same problem. Both $\ue,\tu$ belong to $W^{2,2}_\loc(\overline\Omega)$ because of regularity properties to the Stokes system. Then the vector field $\mathbf{w} \in W_{\text{loc}}^{2,2}(\overline\Omega)$ is a~strong solution to the following system:
\begin{equation}  \label{NSE-www}
	\left\{
	\begin{aligned}
		&- \Delta \mathbf{w} =- (\tu \cdot \nabla) \mathbf{w} -(\mathbf{w} \cdot \nabla) \mathbf{u} - \nabla p   \qquad  \ \ \mbox{ in } \ \ \Omega \, , \\[4pt]
		& \nabla \cdot \mathbf{w} = 0\qquad  \ \ \mbox{ in } \ \ \Omega \, ,  \\[4pt]
		&\mathbf{w} = \0 \ \ \mbox{ on } \ \ \partial \Omega \, , \\[4pt]
		&  	\int_{\sigma_j^0} \mathbf{w} \cdot \ee_1^j = 0 \qquad \forall j \in \{1,\dots,J \} \, .
	\end{aligned}
	\right.
\end{equation}

Then, after an integration by parts, we obtain
\begin{equation}  \label{uni-40mm-}
D(t)\doteq\int\limits_{\Omega^t}|\nabla\we|^2=\int\limits_{\OO^t}(\mathbf{w} \cdot \nabla)\we\cdot\mathbf{u}+
\frac12\int\limits_{\sigma^t}\biggl[\nabla|\we|^2-|\we|^2\tu-2\,\bigl(p+\we\cdot\ue\bigr)\we\biggr]\cdot\n,\qquad \forall t\ge0
\end{equation}
(cf. with~(\ref{ug9})\,). \ By (\ref{ug-est-s1}) and (\ref{uni-23}),
\begin{equation}\label{uni-41mm-}\biggl|\int\limits_{\OO^t}(\mathbf{w} \cdot \nabla)\we\cdot\mathbf{u}\biggr|
	\le \sqrt{\cc}\,\e_0\, D(t)\qquad\forall t\ge 0,\end{equation}
where $\cc=\cc(\Omega)>0$ is the constant in~(\ref{ug-est-s1}).  Assuming $\sqrt{\cc}\,\e_0<\frac12$, we find
\begin{equation}  \label{uni-40}
	D(t)\le\int\limits_{\sigma^t}\biggl[\nabla|\we|^2-|\we|^2\tu-2\,\bigl(p+\we\cdot\ue\bigr)\we\biggr]\cdot\n\qquad \forall t\ge0,
\end{equation}
Now we repeat the arguments from Section~\ref{poet6} with some simplifications.
Put
\begin{equation}\label{uni-41}
	h(t)\doteq \int\limits_{t-1}^t D(\tau)\,d\tau.
\end{equation}
Then by~(\ref{uni-40}),
\begin{equation}\label{uni-ug15}h(t)\le \int\limits_{\OO^{t-1,t}}\biggl[\nabla|\we|^2-|\we|^2\tu-2\,\bigl(p+\we\cdot\ue\bigr)\we\biggr]\cdot\n,\qquad
\forall t\ge 1,
\end{equation}
furthermore,
\begin{equation}\label{uni-ug16}
\biggl|\int\limits_{\OO^{t-1,t}}\biggl[\nabla|\we|^2-|\we|^2\tu-2\,\bigl(p+\we\cdot\ue\bigr)\we\biggr]\cdot\n\biggr|\le c_*
\biggl[ \int\limits_{\OO^{t-1,t}}|\nabla\we|^2+\biggl(\int\limits_{\OO^{t-1,t}}|\nabla\we|^2\biggr)^\frac32\biggr] \qquad\forall t\ge 1,
\end{equation}
with some constant~$c_*>0$, see \cite[proof of Theorem~2.2]{ladyzhenskaya1983determination} for details. Note that the zero flux condition~(\ref{NSE-www})$_4$ again plays the key role in the proof of \eqref{uni-ug16}. The definition~(\ref{uni-41}) implies that
\begin{equation}\label{uni-ug19}
h'(t)=D(t)-D(t-1)=\int\limits_{\OO^{t-1,t}}|\nabla\we|^2\qquad\forall t\ge 1.
\end{equation}
Then from (\ref{uni-ug15})--(\ref{uni-ug19}) we get
\begin{equation}\label{uni-ug20}h(t)\le c_*\biggl[h'(t)+h'(t)^\frac32\biggr]\qquad \forall t\ge 1.\end{equation}
Since $\we\not\equiv\0$, there exists $t_0$ such that $D(t_0-1)=c_0>0$. \ 
Then $c_0\le h(t)$ for all $t\ge t_0$ and, by (\ref{uni-ug20}),
$$
c_1\le h'(t)\qquad \forall t\ge t_0
$$
for some $c_1>0$. Finally, (\ref{uni-ug20}) implies $c_0<h(t)\le \cc_* h'(t)^\frac32$ for all $t\ge t_0$ and some $\cc_*>0$. This differential
inequality easily implies the required~(\ref{uni-25}). So the proofs of Proposition~\ref{th:main5}
and Theorem~\ref{th:main3} are finished. \hfill $\qed$

\section{Two physical models}\label{physics}

We give here two examples of physical models described by system \eqref{NSE}.\par\smallskip
{\bf Suspension bridges.} Within the unbounded strip $S=\R\times(0,1)$ containing a compact obstacle $\Sigma\subset S$,  \eqref{NSE} describes
the (planar) wind flow on the cross-section of the deck of a suspension bridge (in Figure \ref{fig_bridge}, left: the Bosphorus Bridge
in Istanbul). Some bridges, such as the Humber Bridge in England (see \cite[Figure 6-12]{kawada}), have a smoothed irregular
hexagon as cross-section, see Figure \ref{fig_bridge} (right).
\begin{figure}[H]
	\begin{center}
		\includegraphics[height=24mm]{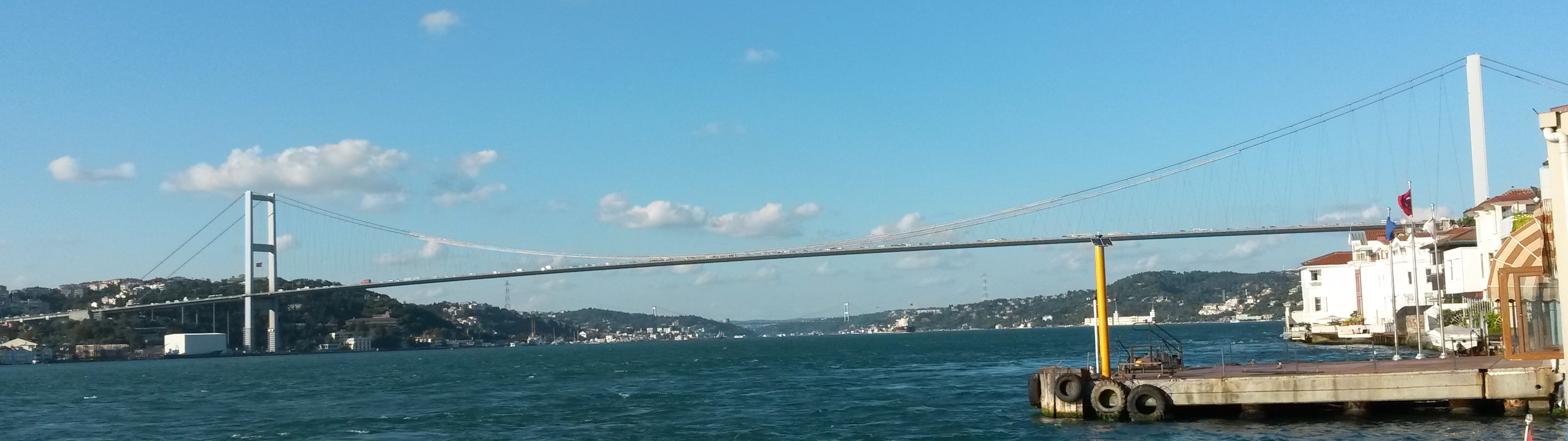}\quad\includegraphics[height=24mm]{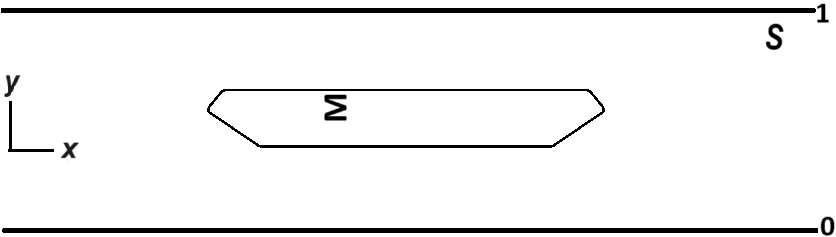}
	\end{center}
	\vspace{-5mm}
	\caption{The Bosphorus suspension bridge (left) and the cross-section of the Humber Bridge (right).}\label{fig_bridge}
\end{figure}

The decks of such bridges are subject to several structural forces (the restoring forces due to the elastic cables+hangers system) and to
the so-called {\it lift force} \cite{langre}. The latter force creates a fluid-structure interaction and generates another force
$\fe\in L^2(\Omega)$ which mainly acts in a compact neighborhood $\Omega_0\subset\Omega$ of the cross-section $\Sigma$.
For this model, $\Omega=S\setminus\Sigma$ so that $I=1$, $J=2$ with $h_1=h_2=1$.
The strip $S$ represents the part of atmosphere going from the Earth (or sea) level to some fixed altitude.
Since the domain is unbounded, the air may be considered incompressible and its flow is then described by \eqref{NSE}$_{1,2}$.
Moreover, the wind velocity is null on the deck which translates into $\aa=0$ in \eqref{NSE}$_3$ and there
exist $V_\pm\in C^1[0,1]$ satisfying
$$
V_\pm(0)=0\, ,\ V_\pm(1)=\overline{V}\ge0\, ,\quad\int_0^1 V_-(y)dy=\int_0^1 V_+(y) dy\, ,\qquad
\quad \lim_{x\to\pm\infty}\ue (x,y)=V_\pm(y)\quad\forall y\in[0,1]\, ,
$$
so that \eqref{ch5} is fulfilled. Typical examples are the Couette flow $V_\pm(y)=\overline{V}y$ or ``half''
Poiseuille flow $V_\pm(y)=\overline{V}y(2-y)$ with $\overline{V}>0$ measuring the wind velocity at altitude $y=1$.
The equilibrium positions of the deck for small Reynolds number $\overline{V}$ were determined in \cite{berchio2024equilibrium,gazspe,gazzola2022connection}
for symmetric flows $V_\pm$ and obstacles $\Sigma$, and in \cite{edofil,bocchi2024measure} in a general non-symmetric framework. Some of these results
were obtained in the rectangle $S_R=(-R,R)\times(0,1)$ (containing $\Sigma$) instead of the unbounded strip $S$: although $\partial R$ is
merely Lipschitz, the presence of (convex) right angles allows to use the classical regularity theory \cite{kellogg1976regularity}.
Then, the fluid domain is $\Omega_R=\{(x,y)\in\Omega;\, |x|<R\}$ and the related boundary-value problem reads
\begin{equation}\label{BVprobR}
	\begin{array}{cc}
		-\Delta \ue +\ue \cdot\nabla \ue +\nabla p=\fe,\quad\nabla\cdot \ue =0\quad\mbox{ in }\Omega_R\\
		\ue _{\Sigma}=\ue (x,0)=\ue (x,1)=\0\quad\forall|x|\le R\, ,\qquad \ue (\pm R,y)=V_\pm(y)\ee_1\quad\forall y\in[0,1].
	\end{array}
\end{equation}

Both \eqref{NSE} and \eqref{BVprobR} are relevant for the wind-bridge problem. Before designing a bridge, several experiments
are performed on a reduced scale in a wind tunnel (WT) where artificial winds "blow air flows" $V_\pm$ at the inlet/outlet,
see Figure \ref{fig_bridge2}.
\begin{figure}[h]
	\begin{center}
		\scalebox{-1}[1]{\includegraphics[width=0.48\textwidth]{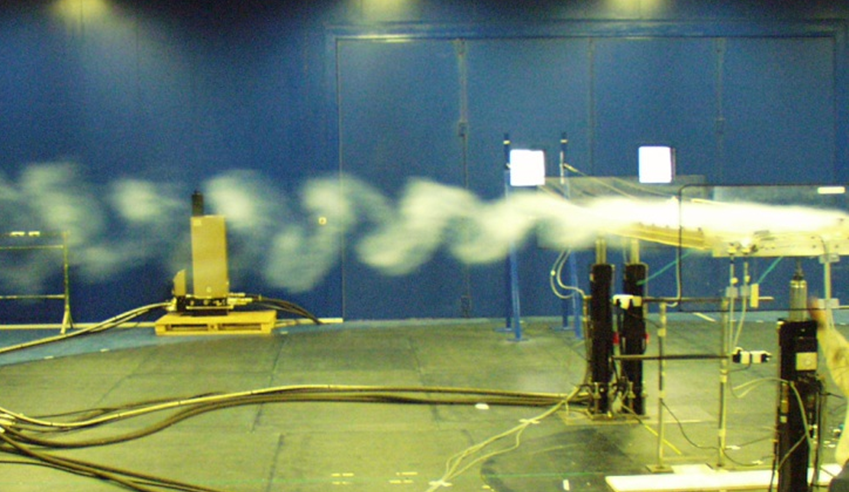}}\qquad
		\includegraphics[scale=0.29]{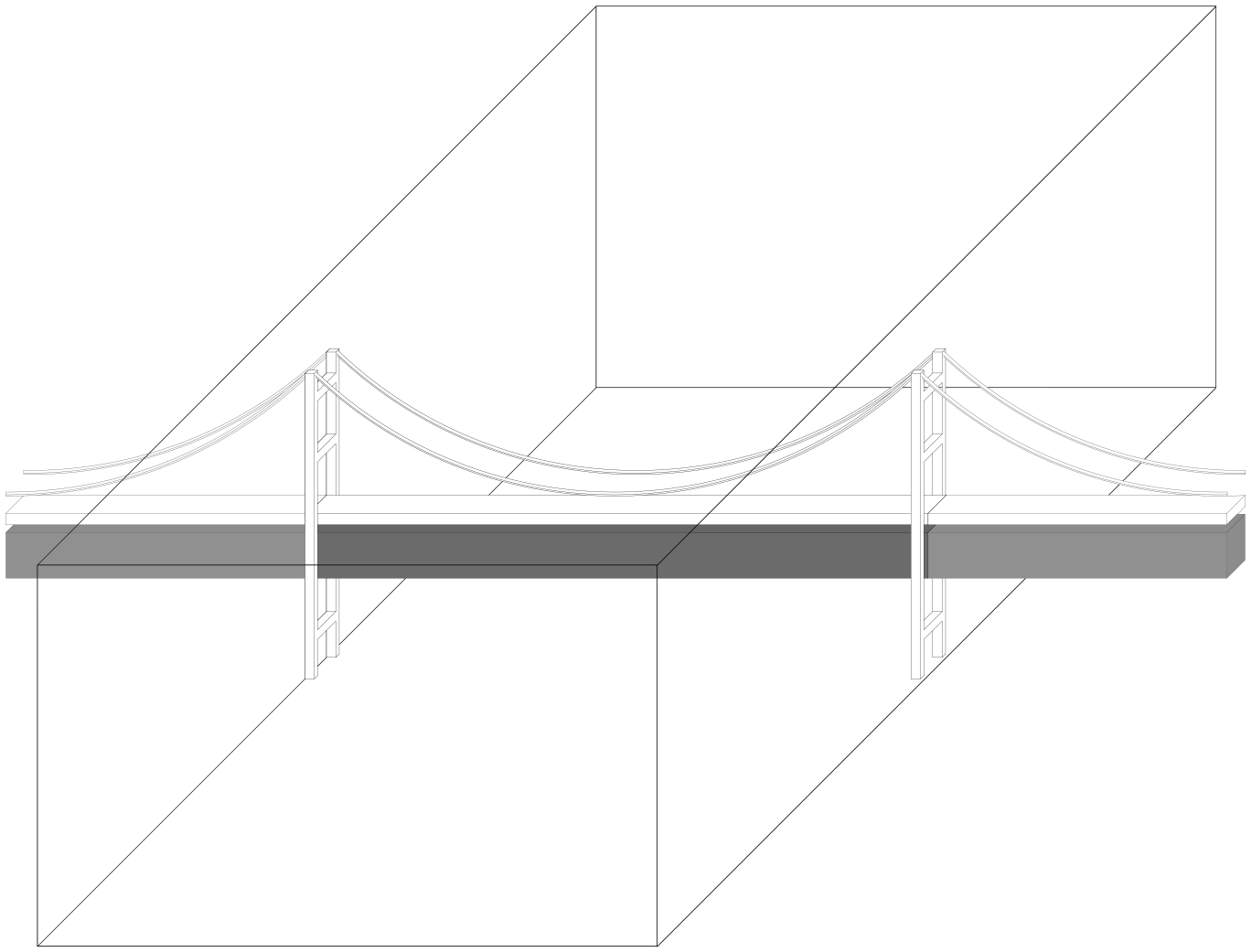}
	\end{center}
	\vspace{-5mm}
	\caption{Flows in the Politecnico WT (left) and sketch of suspension bridge in a WT (right).}\label{fig_bridge2}
\end{figure}
The flow velocities are 0 at the level of the floor ($y=0$) and $\overline{V}>0$ on the ceiling ($y=1$), and may have fairly different
behaviors for $y\in(0,1)$.
This is modeled by \eqref{BVprobR}, with $\Sigma$ being the cross section of the deck, $2R$ is the length of the WT, 1 its height.
The WT experiments aim to forecast the behavior of actual bridges modeled by \eqref{NSE}, which makes sense only if the solutions
to \eqref{BVprobR} converge (in a suitable sense) towards the solutions to \eqref{NSE} as $R\to\infty$, possibly by quantifying the gap between
the solutions in order to bound the discrepancies between (large) WT experiments and real structures.
As we have explained in the Introduction, this problem is still open and its solution is related to the Leray method of invading domains, that we illustrate in Section \ref{proofs-sec}.
If this convergence fails, then WT experiments become useless because they {\it do not} approximate reality!
And, indeed, some behaviors of erected bridges were not forecasted by WT experiments \cite{mdemiranda}. This different behavior appears to be related to the {\it mysterious} statement in~\cite[page 734]{ladyzhenskaya1983determination}.
Theorem~\ref{th:main} and the invading domains procedure merely state that, for the WT experiment, there exists a steady air configuration around
the deck of the bridge with given (possibly large) {\it magnitude} of inflow/outflow without any guarantee that the {\it shape} of the flows at
infinity is preserved. Theorem \ref{th:main2} states that, if the flows are sufficiently small, then regardless of their {\it shape} in WT
experiments, the behavior at infinity always coincides with the corresponding Couette-Poiseuille flow \eqref{CP}.\par\medskip
{\bf Irrigation fountain.} Imagine that we need to irrigate a large agricultural region and that there is
an underground water source placed somewhere in the region. What we can do is to build a fountain in order to allow the water to reach the
Earth level and to dig some channels around the fountain in order to carry the water all over the region, possibly through smaller
``subchannels'', see Figure \ref{fountain}.
\begin{figure}[h]
	\begin{center}
		\includegraphics[width=0.45\textwidth]{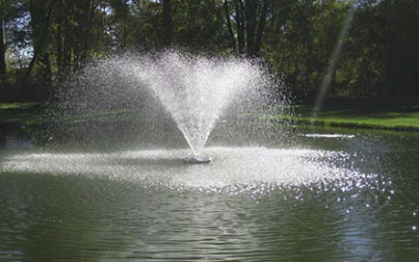} \qquad \includegraphics[width=0.45\textwidth]{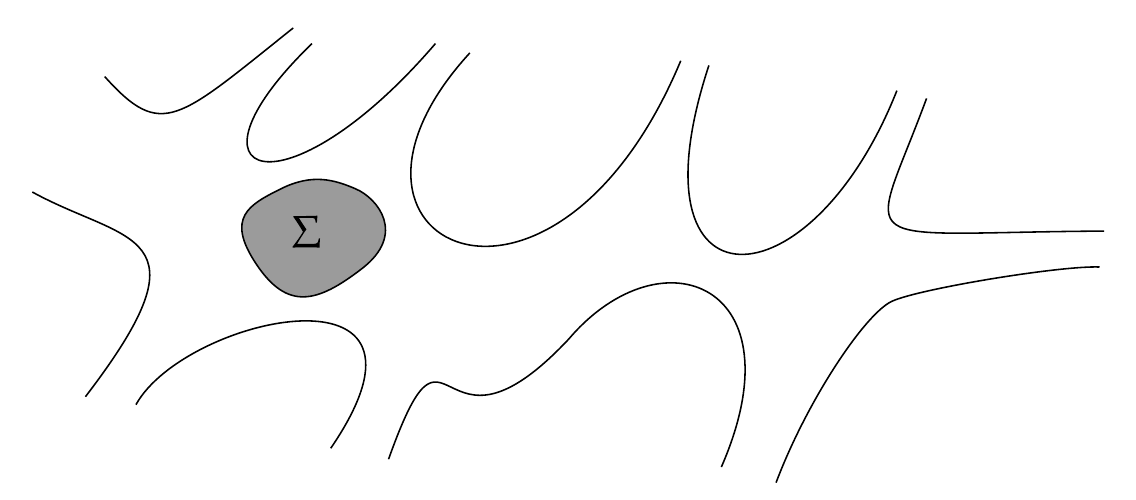}
	\end{center}
	\vspace{-5mm}
	\caption{Irrigation fountain (left) and 9 outflow channels (right).}\label{fountain}
\end{figure}
This is modeled by \eqref{NSE} with $\Sigma$ being the boundary of the fountain (so that $I=1$), while the number of
channels coincides with $J$: not all the channels need to have the same width, it may be $h_1\neq h_2\neq h_3...$.
The water flow is again described by \eqref{NSE}$_{1,2}$. Concerning \eqref{NSE}$_3$, due to viscosity one has $\aa=0$ on the boundary of
the channels, whereas $\aa\cdot\n>0$ on $\Sigma$ (outflow from the boundary of the fountain).
Then the total outflow of the channels (at infinity) needs to balance the inflow and \eqref{ch5} becomes
$$
\sum\limits_{j=1}^J\Fe_j=-\int\limits_{\Sigma}\aa\cdot\n\, .
$$
The most realistic situation occurs when
$$
\Fe_j=\frac{-h_j}{\sum_{m=1}^{J}h_m}\int\limits_{\Sigma}\aa\cdot\n\qquad\forall j=1,...,J,
$$
namely when the flux in each channel is proportional to its width.\par
If the fountain is replaced by a sink (as in the case of a sewer), the same model remains valid but the sign of the flux changes and, then,
$\aa\cdot\n<0$ on $\Sigma$. In fact, the region may also contain multiple fountains and sewers. The target is to plan sufficiently strong
shores $\partial(\Omega\cup\Sigma)$, able to resist to the pressure of given fluxes rates, regardless of the flow shapes. In this case,
Theorem \ref{th:main} and the invading domains technique are much more useful, whereas \eqref{NSE-est} states that one can locally bound
the pressure and the total energy (velocity) of the water.\par\medskip

\appendix\section{Appendix (the construction of the flux carrier).}
\label{appendix}

Here we present an accurate proof of Theorem~\ref{fluxcarrier} concerning the existence of the flux carrier $\mathbf U$ satisfying the required properties. 
In principle, the construction of the solenoidal extension follows the original method proposed by Hopf (see, for example,
\cite[Lemma IX.4.2]{galdi2011introduction} or \cite[Lemma 3.3.1]{korobkov2024steady}\,), but here we require the Leray-Hopf inequality
\eqref{HII} to be satisfied in the (possibly unbounded) rectangles $\Omega^{a,b}_{j}$, over test functions that vanish only on the lower
and upper boundary of $\Omega^{a,b}_{j}$, with $0<a<b$.\par
Firstly, for $\Omega^2$ as in \eqref{Omega_ab}, let $\mathbf{V} \in W^{1,2}(\Omega^2)$ be the unique weak solution
to the Stokes system
\begin{equation}\label{ttO}
\left\{
\begin{aligned}
&- \Delta \mathbf{V}  + \nabla Q = \0 \, , \quad\nabla \cdot \mathbf{V} = 0 \quad\mbox{ in } \ \Omega^2 \, , \\[6pt]
& \mathbf{V} = \ae \quad \mbox{ on } \ \ \partial\Omega\cap \partial \Omega^2 \, , \\[6pt]
& \mathbf{V} = \mathbf{CP}_j \quad \mbox{ on } \ \ {\sigma^{2}_{j}}  \, , \quad \forall j \in \{1,\dots, J\} \, ,
\end{aligned}
\right.
\end{equation}
whose existence is guaranteed by the compatibility condition \eqref{ch5} (recall also \eqref{fluxa}). Notice that the boundary data in \eqref{ttO}$_2$-\eqref{ttO}$_3$ belongs to $W^{3/2,2}(\partial\Omega^2)$ (recall \eqref{nscouette}$_2$). Since $\Gamma$ is of class $\mathcal{C}^{2}$ and the regions $\Omega_j^{0,2}$ are convex polygons for $j \in \{1,\dots, J\}$ (see Figure \ref{fig:5}), standard elliptic regularity results for the Stokes system \cite[Theorem 1]{kellogg1976regularity} in the truncated domain $\Omega^2$ can be applied to deduce that $(\mathbf{V},Q) \in W^{2,2}(\Omega^2) \times W^{1,2}(\Omega^2)$.

Secondly, given $j \in \{1,\dots, J\}$, let $\Psi_{j} \in \mathcal{C}^{\infty}(\overline{\Omega_{j}};\mathbb{R})$ be the stream function associated with the Couette-Poiseuille flow \eqref{CP} in $\Omega_{j}$, meaning that
\begin{equation}\label{str-p1}
\mathbf{CP}_{j}(z) = - \nabla^{\perp} \Psi_{j}(z) \qquad \forall z \in \Omega_{j} \, .
\end{equation}
Such stream function is given explicitly, in local coordinates, by
\begin{equation}\label{str-p2}
	\Psi_j(z) \doteq \frac{A_j}{3} y_{j}^3 + \frac{B_j}{2} y_{j} + b_j^0 y_{j} \qquad \forall z \in \overline{\Omega_{j}} \, .
\end{equation}
On the other hand, let
$$
\delta_{j}(z) \doteq \text{dist}(z, \Gamma^{0}_{j} \cup \Gamma^{1}_{j}) = \min \{y_{j}, h_{j} - y_{j} \} = \dfrac{h_{j} - | h_{j} - 2y_{j} |}{2} \qquad \forall z \in \overline{\Omega_{j}} \, ,
$$
denote the distance between a point $z \in \overline{\Omega_{j}}$ and the walls $\Gamma^{0}_{j} \cup \Gamma^{1}_{j}$. Given any
$\varepsilon \in (0,1]$, as in \cite[Lemma III.6.2]{galdi2011introduction} (see also \cite[Lemma 1]{kravcmar2001weak}) we can construct a
scalar function $\psi_{j}^{\varepsilon} \in \mathcal{C}^{\infty}(\overline{\Omega_{j}})$ (depending only on the vertical variable
$y_{j} \in [0,h_{j}]$) such that:\par
(1) $ | \psi_{j}^{\varepsilon}(z) | \leq 1$ \ \ $\forall z \in \overline{\Omega_{j}} $;

(2) $\psi_{j}^{\varepsilon}(z) = 1$ \ \ if \ \ $\delta_{j}(z) \leq \dfrac{1}{2} e^{-2/\varepsilon}$;
	
(3) $0 \le \psi_{j}^{\varepsilon}(z) \le 1$ \ \ if \ \ $\dfrac{1}{2} e^{-2/\varepsilon} < \delta_{j}(z) < e^{-1/\varepsilon} + \dfrac{1}{2} e^{-2/\varepsilon}$;
	
(4) $\psi_{j}^{\varepsilon}(z) = 0$ \ \ if \ \ $\delta_{j}(z) \geq e^{-1/\varepsilon} + \dfrac{1}{2} e^{-2/\varepsilon}$;
	
(5) $ | \nabla \psi_{j}^{\varepsilon}(z) | \leq \dfrac{\varepsilon}{\delta_{j}(z)}$  \ \ $\forall z \in \Omega_{j}$.

We emphasize that both the distance function $\delta_{j}$ and the cut-off function $\psi_{j}^{\varepsilon}$ depend only on the vertical
variable $y_{j} \in [0,h_{j}]$. We then define the vector field
\begin{equation} \label{vectorfield1rect}
	\mathbf{V}^{\varepsilon}_{j}(z)  \doteq - \nabla^{\perp} ( \psi_{j}^{\varepsilon} \, \Psi_{j})(z) = \psi_{j}^{\varepsilon}(z) \, \mathbf{CP}_{j}(z) - \Psi_{j}(z) \, \nabla^{\perp} \psi_{j}^{\varepsilon}(z) \qquad \forall z \in \overline{\Omega_{j}} \, ,
\end{equation}
which is an element of $\mathcal{C}^{\infty}(\overline{\Omega_{j}})$ verifying, owing to the previous construction, the following:
\begin{equation}\label{vjeps}
	\nabla \cdot \mathbf{V}^{\varepsilon}_{j} =0 \ \ \mbox{ in } \ \ \Omega_{j} \qquad \text{and} \qquad \mathbf{V}^{\varepsilon}_{j} = \mathbf{CP}_{j} \ \ \mbox{ on } \ \ \Gamma^{0}_{j} \cup \Gamma^{1}_{j} \, .
\end{equation}
Furthermore, from \eqref{vectorfield1rect} and the properties of $\psi_{j}^{\varepsilon}$ stated in items $(1)-(5)$ we infer that
\begin{equation} \label{vort0rect}
	\mathbf{V}^{\varepsilon}_{j}(z) = 0 \qquad \forall z \in \Omega_{j} \setminus \Omega^{\varepsilon}_{j} \, ; \qquad | \mathbf{V}^{\varepsilon}_{j}(z) | \leq |\mathbf{CP}_{j}(z)| + \dfrac{\varepsilon}{\delta_{j}(z)} |\Psi_{j}(z)| \qquad \forall z \in \Omega^{\varepsilon}_{j} \, ,
\end{equation}
where
$$
\Omega^{\varepsilon}_{j} \doteq \left\lbrace z \in \Omega_{j} \ \Big\vert \ \delta_{j}(z) < e^{-1/\varepsilon} + \dfrac{1}{2} e^{-2/\varepsilon} \right\rbrace.
$$
Now, given $2<a<b\le+\infty$ and any vector field $\bbe \in  W^{1,2}(\Omega_j^{a,b})$ verifying \eqref{testeta}, an integration by parts shows that
\begin{equation}\label{HI-2rect}
	\int_{ \Omega_j^{a,b}} (\bbe\cdot \nabla)\mathbf{V}^{\varepsilon}_{j} \cdot \bbe = - \int_{ \Omega_j^{a,b}} (\bbe\cdot \nabla) \bbe \cdot \mathbf{V}^{\varepsilon}_{j} + \int_{\partial \Omega_j^{a,b}} (\bbe\cdot \n) (\mathbf{V}^{\varepsilon}_{j} \cdot \bbe) \, .
\end{equation}
Notice that $\mathbf{V}^{\varepsilon}_{j}(z) \cdot \mathbf{e}^{j}_{2} = 0$ for every $z \in \overline{\Omega_{j}}$, see \eqref{vectorfield1rect}.
Setting $\eta^{j}_{1} \doteq \bbe \cdot \mathbf{e}^{j}_{1}$ in $\Omega_j^{a,b}$, and recalling that $\bbe$ vanishes on
$\Gamma^{0}_{j} \cup \Gamma^{1}_{j}$, it follows from the Divergence Theorem that
\begin{equation}
	\int_{\partial \Omega_j^{a,b}} (\bbe\cdot \n) (\mathbf{V}^{\varepsilon}_{j} \cdot \bbe) = \int_{\partial \Omega_j^{a,b}} |\eta^{j}_{1}|^{2} (\mathbf{V}^{\varepsilon}_{j} \cdot \n) = \int_{\Omega_j^{a,b}} \nabla \cdot ( |\eta^{j}_{1}|^{2} \, \mathbf{V}^{\varepsilon}_{j} ) = 2 \int_{\Omega_j^{a,b}} \eta^{j}_{1} \nabla \eta^{j}_{1} \cdot \mathbf{V}^{\varepsilon}_{j} \, ,
\end{equation}
which, once inserted into \eqref{HI-2rect}, yields
\begin{equation}\label{HI-3rect}
	\int_{ \Omega_j^{a,b}} (\bbe\cdot \nabla)\mathbf{V}^{\varepsilon}_{j} \cdot \bbe = - \int_{ \Omega_j^{a,b}} (\bbe\cdot \nabla) \bbe \cdot \mathbf{V}^{\varepsilon}_{j} + 2 \int_{\Omega_j^{a,b}} \eta^{j}_{1} \nabla \eta^{j}_{1} \cdot \mathbf{V}^{\varepsilon}_{j} \, .
\end{equation}
In the remaining part of the proof, $C > 0$ will denote a generic constant that depends exclusively on $\{h_{j},b^{0}_{j},b^{1}_{j},\mathcal{F}_{j} \}$, but that may change from line to line. Setting
$$
\Omega^{a,b,\varepsilon}_{j} \doteq \Omega_j^{a,b} \cap \Omega^{\varepsilon}_{j} \, ,
$$
from the Poincaré inequality we infer
\begin{equation} \label{vort03}
	\|\bbe \|_{L^{2}(\Omega^{a,b,\varepsilon}_{j})}  \leq C \varepsilon \|\nabla  \bbe \|_{L^{2}(\Omega^{a,b,\varepsilon}_{j})} \leq C \varepsilon \|\nabla  \bbe \|_{L^{2}(\Omega^{a,b}_{j})}  \, .
\end{equation}
We apply the H\"older inequality, together with \eqref{vort0rect}-\eqref{vort03}, to obtain the bound
\begin{equation} \label{vort3}
	\begin{aligned}
		& \left| \int_{\Omega_j^{a,b}} (\bbe \cdot \nabla)\bbe \cdot \mathbf{V}^{\varepsilon}_{j} \right| = \left| \int_{\Omega^{a,b,\varepsilon}_{j}} (\bbe \cdot \nabla)\bbe \cdot \mathbf{V}^{\varepsilon}_{j}  \right| \leq \|\nabla  \bbe \|_{L^{2}(\Omega^{a,b,\varepsilon}_{j})} \, \| \, |\bbe| \, |\mathbf{V}^{\varepsilon}_{j}| \, \|_{L^{2}(\Omega^{a,b,\varepsilon}_{j})} \\[6pt]
		& \hspace{-4mm} \leq \|\nabla  \bbe \|_{L^{2}(\Omega^{a,b}_{j})} \left\| |\bbe| \, |\mathbf{CP}_{j}| + \varepsilon \, |\Psi_{j}| \, \dfrac{|\bbe|}{\delta_{j}} \right\|_{L^{2}(\Omega^{a,b,\varepsilon}_{j})} \\[6pt]
		& \hspace{-4mm} \leq \|\nabla  \bbe \|_{L^{2}(\Omega^{a,b}_{j})}  \left( \| \mathbf{CP}_{j} \|_{L^{\infty}(\Omega_{j})} \|\bbe \|_{L^{2}(\Omega^{a,b,\varepsilon}_{j})}  + \varepsilon \, \| \Psi_{j} \|_{L^{\infty}(\Omega_{j})} \left\| \dfrac{\bbe}{\delta_{j}} \right\|_{L^{2}(\Omega^{a,b}_{j})}  \right) \\[6pt]
		& \hspace{-4mm} \leq C \varepsilon \left( \|\nabla  \bbe \|_{L^{2}(\Omega^{a,b}_{j})}  + \left\| \dfrac{\bbe}{\delta_{j}} \right\|_{L^{2}(\Omega^{a,b}_{j})}  \right) \|\nabla  \bbe \|_{L^{2}(\Omega^{a,b}_{j})} \, .
	\end{aligned}
\end{equation}
Since $\bbe$ vanishes on $\Gamma^{0}_{j} \cup \Gamma^{1}_{j}$, by the Hardy inequality
$$
\int_{0}^{h_{j}} \dfrac{|\bbe(x_{j},y_{j})|^{2}}{(h_{j} - | h_{j} - 2y |)^2} \, dy_{j} \leq
C \int_{0}^{h_{j}} \left| \dfrac{\partial \bbe}{\partial y}(x_{j},y_{j}) \right|^{2} \, dy_{j} \qquad \forall x_{j} \in (a,b)
$$
holds, from where it directly follows the estimate
\begin{equation} \label{vort4}
	\left\| \dfrac{\bbe}{\delta_{j}} \right\|_{L^{2}(\Omega^{a,b}_{j})} \leq C \left\| \nabla \bbe \right\|_{L^{2}(\Omega^{a,b}_{j})} \, .
\end{equation}
Inserting \eqref{vort4} into \eqref{vort3} gives
\begin{equation} \label{vort5}
	\left| \int_{\Omega^{a,b}_{j}} (\bbe\cdot \nabla)\bbe\cdot \mathbf{V}^{\varepsilon}_{j} \right| \leq C \varepsilon \, \| \nabla\bbe \|_{L^{2}(\Omega^{a,b}_{j})}^2 \, .
\end{equation}
In the same way one can show that
\begin{equation} \label{vort6}
	\left| \int_{\Omega^{a,b}_{j}} \eta^{j}_{1}  \nabla \eta^{j}_{1} \cdot \mathbf{V}^{\varepsilon}_{j} \right| \leq C \varepsilon \, \| \nabla\bbe \|_{L^{2}(\Omega^{a,b}_{j})}^2 \, ,
\end{equation}
and therefore, owing to \eqref{HI-3rect} and recasting the value of $\varepsilon$ (recall that the constant $C > 0$ entering \eqref{vort5}-\eqref{vort6} is independent of $a$ and $b$) we obtain the Leray-Hopf inequality
 \begin{equation}\label{HIv}
	\left| \int_{  \Omega_j^{a,b}} (\bbe\cdot \nabla)\mathbf{V}^{\varepsilon}_{j} \cdot \bbe \right| +
	\left| \int_{ \Omega_j^{a,b}} (\bbe\cdot \nabla)\bbe\cdot \mathbf{V}^{\varepsilon}_{j} \right| \leq \varepsilon \, \| \nabla\bbe \|_{L^{2}(  \Omega_j^{a,b})}^2 \, .
\end{equation}

Thirdly, given $j \in \{1,\dots, J\}$, we smoothly connect the vector fields $\mathbf{V}$ and $\mathbf{V}^{\varepsilon}_{j}$ in the following
fashion. Let $\phi_{j} \in W^{3,2}(\Omega_j^{0,2})$ be a~stream function associated with $\mathbf{V}$ in the rectangle $\Omega_j^{0,2}$:
$$
\mathbf{V}(z) = - \nabla^{\perp} \phi_{j}(z)  \qquad \forall z \in \Omega^{0,2}_{j} \, .
$$
Since $\mathbf{V} = \mathbf{CP}_{j} = \mathbf{a}$ on $\Gamma \cap \partial \Omega^{0,2}_{j}$, see \eqref{ttO}$_2$, there
exist $c^{0}_{j}, c^{1}_{j} \in \mathbb{R}$ such that
$$
\phi_{j} = \Psi_{j} + c^{0}_{j} \ \ \text{ on } \ \ \Gamma_j^{0} \cap \partial \Omega^{0,2}_{j} \qquad \text{and} \qquad \phi_{j} = \Psi_{j} + c^{1}_{j} \ \ \text{ on } \ \ \Gamma_j^{1} \cap \partial \Omega^{0,2}_{j} \, .
$$
Since the stream function $\phi_{j}$ can be modified up to an additive constant, we can assume $c^{0}_{j} = 0$. On the other hand,
the boundary conditions \eqref{ttO}$_3$, together with the flux condition \eqref{fluxcp0}, imply that
$$
\int_{\sigma_j^2} \dfrac{\partial}{\partial \n^{\perp}} (\phi_{j} - \Psi_{j}) = 0 \, ,
$$
thereby yielding $c^{1}_{j} = 0$, that is, 
\begin{equation}\label{strf-coincide}
\phi_{j}(z) \equiv \Psi_{j}(z)\equiv \psi_{j}^{\varepsilon}(z)  \Psi_{j}(z) \qquad\forall z\in\Gamma \cap \partial \Omega^{0,2}_{j}, \quad j=1,\dots, J
\end{equation}
(see~(\ref{str-p1})--(\ref{str-p2})\,). Next, we introduce
a cut-off function $\zeta \in \mathcal{C}^{\infty}(\mathbb{R}^2; [0,1])$ such that
$$
\zeta \equiv 1 \ \ \text{in a neighborhood of} \ \ \overline{\Omega_{0}} \, ; \qquad \zeta \equiv 0 \ \ \text{in a neighborhood of} \ \ \overline{\Omega \setminus \Omega^2} \, .
$$
Then, we define the vector field $\mathbf{U} \in W^{2,2}_{\loc}(\overline\Omega)$ by
\begin{equation}
  \mathbf{U} \doteq \begin{cases}
  \mathbf{V} & \ \mathrm{ in } \ \ \Omega_0  \\[4pt]
  - \nabla^{\perp} (\zeta \, \phi_{j} + (1 - \zeta) \psi_{j}^{\varepsilon} \, \Psi_{j} ) & \ \mathrm{ in }  \ \ \Omega_j \, , \ \ \forall j \in \{ 1,\dots,J\} \, .
  \end{cases}
\end{equation}
Let us check that $\mathbf{U}$ satisfies the properties stated in \eqref{vecpsi}-\eqref{HII}. A straightforward calculation shows that $\mathbf{U}$ is divergence-free in $\Omega$. From \eqref{ttO}$_2$ and from \eqref{vjeps},  \eqref{strf-coincide} we also deduce 
that $$\mathbf{U} \equiv \mathbf{a}\qquad\mbox{ on }\ \partial \Omega.$$ Given $j \in \{ 1,\dots,J\}$, it can be seen
as in \eqref{fluxcp0}, that
$$
\int_{\sigma_j^0} \mathbf{U} \cdot \ee_1^j = \int_{\sigma_j^0} \mathbf{V} \cdot \ee_1^j = \int_{\sigma_j^0} \mathbf{CP}_{j} \cdot \ee_1^j = \Fe_j \, ,
$$
so that the flux condition in \eqref{vecpsi} is fulfilled. Since $\mathbf{U} = \mathbf{V}^{\varepsilon}_{j}$ in
$\Omega_{j}^{2,\infty}$, for every $j \in \{ 1,\dots,J\}$, the uniform Leray-Hopf inequality \eqref{HII} is obtained directly from \eqref{HIv}.
Finally, the local uniform bound \eqref{uboundz2} follows by construction, since the vector field $\mathbf{V}^{\varepsilon}_{j}$ depends only
on the vertical variable $y_{j} \in [0,h_{j}]$ in $\Omega_{j}^{2,\infty}$, for every $j \in \{ 1,\dots,J\}$, see again \eqref{vectorfield1rect}.
\hfill$\qed$.

\section{Appendix (entire solutions to the  Euler equations).} \label{poet2}
The purpose of this section is to prove Theorem~\ref{Euler=zero}, by going through a careful study of fine properties of stationary solutions to the Euler equations in the whole $\mathbb{R}^2$. In the sequel, given a vector field $\ve\in W_{0, \sigma}^{1,2}(\Omega)$, we say that
\begin{center}
	{\it {\rm (E)} holds} if $\ve$ satisfies the Euler system~(\ref{Euler0}) with~(\ref{310});\\[4pt]
	{\it {\rm (E-NS)} holds} if $\ve$ is obtained as a weak limit of solutions $(\widehat{\mathbf{w}}_k)_{k \in \mathbb{N}} \subset W_{0, \sigma}^{1,2}(\Omega)$ to~(\ref{NS-Ek}).
\end{center}

The next statement was proved by Kapitanskii--Pileckas \cite{kapitanskii1983spaces} and Amick \cite{amick1984existence}, see also \cite[Lemma~3.4.1]{korobkov2024steady}.

\begin{proposition}
\label{kmpTh2.3'} {\sl Let the conditions {\rm (E)} be fulfilled.
Then
\begin{equation} \label{bp2} \forall i\in\{1,\dots,I\} \ \exists\, \check q_i\in\R:\quad
q(z)\equiv \check q_i\quad\mbox{for }\Ha^1-\mbox{almost all }
z\in\partial\Sigma_i.\end{equation}
\begin{equation} \label{bp3} \forall j\in\{1,\dots,J\} \ \exists\, \tilde q_j\in\R:\quad
q(z)\equiv \tilde q_j\quad\mbox{for }\Ha^1-\mbox{almost all }
z\in\Gamma_j.\end{equation} }
\end{proposition}

For convenience, we extend the function $\ve$ to the whole plane $\R^2$
by the natural identity
\begin{equation}
\label{axc10.10} \ve(z)\equiv\0\qquad\forall z\in \R^2\setminus\Omega.
\end{equation}
Simultaneously, we extend the pressure $q$ to the whole plane $\R^2$ by the corresponding constants
from (\ref{bp2})--(\ref{bp3}). The extended functions inherit the properties of the
original ones. Namely, we have
\begin{equation} \label{nablav2-1}
	\int\limits_{\R^2} |\nabla\ve|^2 \le 1,
\end{equation}
and $q\in W^{1,s}_\loc(\R^2)$, $s\in(1,2)$, and the Euler equations~(\ref{Euler0}) are
fulfilled almost everywhere in~$\R^2$. Hence, the pair
$(\ve,q)$ is a weak (= Sobolev) solution to the system~(\ref{Euler0}) {\it in the whole plane}.

Due to \eqref{nablav2-1}, \eqref{axc10.10} as well as  the geometry of the domain~$\Omega$, we derive that
\begin{equation} \label{bp4}
\ve\in L^2(\R^2),
\end{equation}
and then
\begin{equation} \label{hp4-2}
	\nabla q\in L^1(\R^2).
\end{equation}
In fact, using the classical div-curl lemma we can deduce more on the regularity of the pressure, as the following assertion holds.

 \begin{proposition}[see Theorem~4.6 in \cite{KPR20}]
\label{pp} {\sl Let the conditions {\rm (E)} be fulfilled. Then
\begin{equation} \label{co-1}
\nabla^2 q\in L^1(\R^2), \qquad\nabla q\in L^2(\R^2).
\end{equation}
In particular, the function~$q$ is continuous in $\R^2$ and convergent at infinity, i.e.,
\begin{equation} \label{co-2}
\exists \lim\limits_{|z|\to\infty}q(z)\in \R.
\end{equation}}
\end{proposition}

Without loss of generality, we assume that
 \begin{equation} \label{co-3}
 \lim\limits_{|z|\to\infty}q(z)=0.
\end{equation}
Then the identities~(\ref{bp3}) can be simplified as
\begin{equation} \label{bp3-0}
q(z)\equiv 0,\qquad\forall z\in\Gamma_1\cup\dots\cup\Gamma_J.\end{equation}
By \eqref{hp4-2} and \eqref{bp3-0}, there holds that
 \begin{equation} \label{co-5}
q\in L^1(\R^2).
\end{equation}
The Bernoulli pressure 
$$\Phi\doteq q+\frac12|\ve|^2$$ 
plays a key role in proving the identity~(\ref{pp8}). Recall that
\begin{equation}  \label{pp9}
\nabla\Phi=\omega\ve^\bot=\omega\,(-v_2,v_1),
\end{equation}
where $\omega(z)\doteq -\partial_1v_2+\partial_2v_1$ is the corresponding vorticity.
By construction, $\Phi\in W^{1,1}(\R^2)$ and
\begin{equation} \label{bf0}
\Phi(z)\equiv 0\qquad\forall z\in\Gamma_1\cup\dots\cup\Gamma_J.
\end{equation}

The second equality (\ref{Euler0}$_2$)  (which is fulfilled
in the whole plane~$\R^2$ after the
above extension agreement~(\ref{axc10.10})\,) implies the existence of a~continuous stream
function $\psi\in W^{2,2}_{\loc}(\R^2)$ such that
$$\nabla\psi=\ve^\bot,\quad\mbox{i.e.}\quad\frac{\partial\psi}{\partial x_1}=-v_2,\quad\frac{\partial\psi}{\partial x_2}=v_1.$$
Then, due to \eqref{Euler0}$_2$ and \eqref{bp4}, we have
 \begin{equation} \label{bfs1}
\nabla\psi\in L^2(\R^2),\qquad\nabla^2\psi\in L^2(\R^2).
\end{equation}
We choose the stream function in such a~way that
\begin{equation} \label{bfs0}
\psi(z)\equiv 0,\qquad\forall z\in\Gamma_1\cup\dots\cup\Gamma_J.\end{equation}
Then, from (\ref{bfs1}) and from the geometry of the~domain $\Omega$ we deduce that
\begin{equation} \label{bfs00}
\psi(z)\to0\quad\mbox{ as }|z|\to\infty.\end{equation}

Let us formulate some regularity results concerning the considered
functions $\Phi$ and $\psi$.

\begin{lemma}[see, e.g., Theorem 5.4.2 in \cite{korobkov2024steady}]\label{kmpTh2.1}
{\sl If conditions {\rm (E)} are satisfied, then $\psi\in C(\R^2)$
and there exists a set $A_{\ve}\subset \R^2$ such that

 {\rm (i)}\ $ \mathscr{H}^1(A_{\ve})=0$;

{\rm (ii)} all points  $x\in\R^2\setminus A_{\ve}$ are the Lebesgue points for the functions $\ve$ and $\Phi$;

{\rm  (iii) } for every  $\varepsilon >0$ there exists a set
$U\subset \mathbb{R}^2$ with
$\mathscr{H}^1_\infty(U)<\varepsilon$ such that $A_{\ve}\subset
U$ and the functions $\ve, \Phi$ are continuous in $\R^2\setminus
U$.}
\end{lemma}

The next assertion, obtained in~\cite[Theorem~3.2]{kpr}, is
another key tool in our approach.

\medskip

\begin{proposition}[{\bf Bernoulli Law for Sobolev solutions}, see, e.g., Section~5.4.2 in~\cite{korobkov2024steady}]
\label{kmpTh2.2} {\sl Let the conditions~{\rm (E)} be valid. Then
there exists a set $A_\ve\subset \R^2$ with $\Ha^1(A_\ve)=0$, such
that for any compact connected set
$K\subset\R^2$ the following property holds: if
\begin{equation}
\label{2.4} \psi\big|_{K}=\const,
\end{equation}
then
\begin{equation}
\label{2.5'}  \Phi(z_1)=\Phi(z_2) \quad\mbox{for
 all \,}z_1,z_2\in K\setminus A_{\bf v}.
\end{equation}
}
\end{proposition}

We may assume without loss of generality that the
sets $A_\ve$ from Lemma~\ref{kmpTh2.1} and Proposition~\ref{kmpTh2.2}
are the same. Proposition~\ref{kmpTh2.2} suggests to seek some information concerning the level sets of the~stream function~$\psi$. This information is provided by the next Morse--Sard--type theorem have been proved by Bourgain,
Korobkov and Kristensen \cite{BKK13}.

\begin{proposition}[see, e.g., Theorem 1.6.2 in~\cite{korobkov2024steady}]
\label{kmpTh1.1}{\sl  Let the conditions~{\rm (E)} be valid. Then for $\mathscr{H}^1$--almost all $t\in\psi(\overline{\Omega})$ the preimage
$\psi^{-1}(t)$ is a finite disjoint family of $C^1$--curves $S_\mu$,
$\mu=1, 2, \ldots, N_t$. Moreover, each $S_\mu$ is a cycle in
$\Omega$, i.e., $S_\mu\subset\Omega$ is homeomorphic
to the unit circle, and the identity $S_\mu(t)\cap A_\ve=\emptyset$ holds. }
\end{proposition}

Most of the previous assertions were formulated under the general conditions~(E), i.e., for $W^{1,2}$ solutions to~(\ref{Euler0}). Now we have to exploit the fact that the~solution $\ve$ is obtained as limit of Navier-Stokes solutions to~(\ref{NS-Ek}).
The next assertion was obtained in~\cite[Lemma~3.3]{kpr}, see also~\cite[Lemma~5.4.8]{korobkov2024steady} for a~more detailed explanation. 

\begin{proposition}
\label{regc-ax} {\sl Assuming the conditions~{\rm (E-NS)}, there exists
a~subsequence~$\Phi_{k_\ell} \doteq \widehat{p}_{k_\ell}+ \frac12 |\widehat{\ue}_{k_\ell}|^2$ such that for almost all~$t\in\psi(\overline{\Omega})$  the sequence $\Phi_{k_\ell}|_{S_{\mu}(t)}$ converges uniformly to
the constant~$\Phi(S_{\mu}(t))$ on every cycle $S_{\mu}(t)$ from Proposition~\ref{kmpTh1.1}, $\mu=1,\dots,N_t$.}
\end{proposition}

Note that the identity $\Phi|_{S_{\mu}(t)}\equiv\const$ follows directly from the Bernoulli Law (Proposition~\ref{kmpTh2.2}).  Below we assume without loss of generality that {\bf the subsequence
$\Phi_{k_\ell}$ coincides with the whole sequence $\Phi_{k}$.}
Furthermore, we will refer to   $S_{\mu}(t)$  which
satisfy the~assertion of Proposition~\ref{regc-ax} as {\bf regular cycles}.

To proceed, we consider the following two
separate cases.

\smallskip

(a) The essential supremum of $\Phi$ is not attained at
infinity:
\begin{equation}\label{as-prev1}
\esssup\limits_{z\in\Omega}\Phi(z)>0.
\end{equation}

(b) The essential supremum of $\Phi$ is attained at
infinity:
\begin{equation}\label{as1-axxx} \esssup\limits_{z\in\Omega}\Phi(z)=0.
\end{equation}

\subsection*{The case $\esssup\limits_{z\in\Omega}\Phi(z)>0$.}
\label{EPcontr-axx}

Suppose that the inequality~(\ref{as-prev1}) is fulfilled. Then, using the results of the last subsection (especially, Propositions~\ref{kmpTh2.2}--\ref{regc-ax}), we can separate infinity by regular stream lines from the consideration, thus, this case will be easily reduced to the case of bounded domains considered in~\cite{korobkov2015solution} (see also  \cite[Chapter~5]{korobkov2024steady} for a~more detailed exposition). Indeed, using the~continuity properties of the Bernoulli pressure (see, e.g., Lemma~\ref{kmpTh2.1}$_3$), it is easy to prove the existence of a regular cycle $S\subset\Omega$ such that
\begin{equation}
\label{sm1}  \psi(S)=t\ne0,
\end{equation}
\begin{equation}
\label{sm2}  \Phi_k|_S\mbox{ \ converges uniformly to
the constant~}t_0=\Phi(S)>0,\end{equation}
and
\begin{equation}
\label{sm3}  \esssup\limits_{z\in\Omega_S}\Phi(z)>t_0,
\end{equation}
where we denote by $\Omega_S$ the domain surrounded by the cycle~$S$, i.e., $\partial\Omega_S=S$ (note, that $\bar\Omega_S$ is a bounded set diffeomorphic  to the closed unit disc).

Now we can repeat almost ``word by word" the arguments from~\cite[Section~5.5]{korobkov2024steady} to get the contradiction. The only differences are as follows: now the bounded domain $\Omega'\doteq \Omega\cap\Omega_S$ plays the role of $\Omega$ from \cite[Section~5.5]{korobkov2024steady}, and, respectively, the regular cycle $S$ plays the role of the external boundary component $\Gamma_0$ from \cite[Section~5.5]{korobkov2024steady}.

\subsection*{The case $\esssup\limits_{z\in\Omega}\Phi(z)=0$.}\label{Euler-contr-2}

Suppose now that (\ref{as1-axxx}) holds, i.e., the maximum of
$\Phi$ is attained at
infinity. It can not be reduced to the case of bounded domains, because the unbounded boundary components $\Gamma_j$ with maximum value $\Phi|_{\Gamma_j}\equiv0$ can not be excluded from the consideration anymore, so we are not able to apply the technique from~\cite{korobkov2015solution} and \cite[Chapter~5]{korobkov2024steady}. But, fortunately, this case can be covered by a~much more elementary argument: really, it turns out, that for this case assumption~(E) already imply the required identity~(\ref{pp8}). 
The following new result is of independent interest for studying solutions to the Euler system in unbounded plane domains.

\begin{lemma}
\label{lemma-E-0} {\sl Let the formulae {\rm(\ref{Euler0})}, {\rm(\ref{310}}, {\rm(\ref{bf0})}, and {\rm(\ref{as1-axxx})} be fulfilled, then 
$\ve\equiv\0$.
}
\end{lemma}
\medskip
Indeed, suppose that this is not true, i.e., that the above assumptions are valid, but
$ \ve\ne\const$,
in other words, $\ve$ is not identically zero. Then, we must have $\Phi\ne\const$ as well (otherwise by \eqref{pp9} the vorticity  is identically zero, and $\ve$ is a nonconstant harmonic vector field on the whole plane going to zero at infinity, a~contradiction). Therefore,
 \begin{equation}\label{i-as1} \essinf\limits_{z\in\Omega}\Phi(z)<0.
\end{equation}
The main tool here is the following well-known elementary identity for Euler solutions to~(\ref{Euler0})$_{1-2}$ in two dimensions:
 \begin{equation}\label{i-as2} \div\biggl(q z+(\ve\cdot z)\ve\biggr)=2\Phi
\end{equation}
(this identity was used, in particular, in \cite{kpr,korobkov2015solution}, and in many other papers).
Here, $z = (z_1, z_2)$ and $\ve\cdot z=v_1z_1+v_2z_2$ stands for the inner product. Recall that
 \begin{equation}\label{i-as3} \ve|_{\partial\Omega}\equiv\0,
\qquad q(z)\equiv 0,\quad\forall z\in\Gamma_1\cup\dots\cup\Gamma_J,
\end{equation}
and
\begin{equation}\label{i-as5} q(z)\equiv \check q_i\le0,\quad\forall
z\in\Sigma_i, \ i\in\{1,\dots,I\},\end{equation}
the last claim follows from~(\ref{bp2}) and~(\ref{as1-axxx}).

Recall the definition of~$\Omega^t$ from~(\ref{Omega_ab}). 
Integrating (\ref{i-as2}) over $\Omega^t$, we get
\begin{align}\label{i-as8}
	2\int\limits_{\Omega^t}\Phi&=
\sum\limits_{i=1}^I\check q_i\int\limits_{\partial\Sigma_i}z\cdot\n+
\sum\limits_{j=1}^J\int\limits_{\sigma^t_j}\bigl(q z\cdot\n+ v_{1}^{(j)} \ve\cdot z\bigr)\,dy_j \nonumber \\
 &=-2\sum\limits_{i=1}^I\check q_i\,|\Sigma_i|+\sum\limits_{j=1}^J\int\limits_{\sigma^t_j}\bigl(q z\cdot\n + v_{1}^{(j)} \ve\cdot z \bigr)\,dy_j.
\end{align}
Here, $\n$ is the~outward normal vector with respect to~$\partial\Omega^t$, and  $|\Sigma_i|$ is the~measure of the~bounded domains~$\Sigma_i$,  and $v_1^{(j)}$ is the first component of $\mathbf{v}$ in the $(x_j, y_j)$-coordinates.
Therefore, there holds that
\begin{align}\label{i-as9}
	\sum\limits_{j=1}^J\int\limits_{\sigma^t_j}\bigl(|q|+\ve^2\bigr)\,dy_j&\ge-\frac{c_0}t
\sum\limits_{j=1}^J\int\limits_{\sigma^t_j}\bigl(q z\cdot\n + v_{1}^{(j)} \ve\cdot z\bigr)\,dy_j \nonumber \\
&=-
\frac{2c_0}t
\biggl(\int\limits_{\Omega^t}\Phi+
\sum\limits_{i=1}^I\check q_i\, |\Sigma_i|\biggr)\overset{(\ref{as1-axxx}),(\ref{i-as5})}\ge\frac{c_1}t
\end{align}
with some positive constants~$c_0,c_1$. But the last inequality definitely contradicts the fact that
$|q|+\ve^2\in L^1(\R^2)$ (see (\ref{bp4}), (\ref{co-5})\,). We have proved Lemma~\ref{lemma-E-0}, and this finishes the proof of Theorem~\ref{Euler=zero}.  $\qed$

\medskip

\noindent
{\bf Acknowledgements.} The research of Filippo Gazzola is supported by the grant \textit{Dipartimento di Eccellenza 2023-2027},
issued by the Ministry of University and Research (Italy) and is a~part of the PRIN project 2022 ``Partial differential equations and related geometric-functional inequalities'', financially supported by the EU, in the framework of the ``Next Generation EU initiative''; the first author is also partially suppported by INdAM. The research of Gianmarco Sperone is currently supported by the \textit{Chilean National Agency for Research and Development} (ANID) through the \textit{Fondecyt Iniciación} grant 11250322. The research of Xiao Ren is supported by  the \textit{China Postdoctoral Science
Foundation} under Grant Number BX20230019.
\par\smallskip
\noindent
{\bf Data availability statement.} Data sharing not applicable to this article as no datasets were generated or analyzed during the current study.
\par\smallskip
\noindent
{\bf Conflict of interest statement}.  The authors declare that they have no conflict of interest.

\phantomsection
\addcontentsline{toc}{section}{References}
\bibliographystyle{abbrv}

\vspace{5mm}

\small{
\noindent
\hspace{0.1mm}
\begin{minipage}{140mm}
	\textbf{Filippo Gazzola}\\
	Dipartimento di Matematica\\
	Politecnico di Milano\\
	Piazza Leonardo da Vinci 32\\
	20133 Milan - Italy\\
	E-mail: filippo.gazzola@polimi.it
\vspace{0.5cm}	
\end{minipage}
\newline
\vspace{0.5cm}
\noindent
\begin{minipage}{100mm}
	\textbf{Mikhail Korobkov}\\
	School of Mathematical Sciences\\
	Fudan University\\
	Handan Road 220\\
	200433 Shanghai - China\\
	and\\
	Sobolev Institute of Mathematics\\
	Siberian Branch of the Russian Academy of Sciences\\
	Akademika Koptyuga Prospekt 4\\
	630090 Novosibirsk - Russia\\
	E-mail: korob@math.nsc.ru
\end{minipage}
\newline
\vspace{0.5cm}
\begin{minipage}{100mm}
	\textbf{Xiao Ren}\\
	Center for Applied Mathematics\\
Fudan University\\
Handan Road 220\\
	200433 Shanghai - China\\
	E-mail: xiaoren18@fudan.edu.cn
\end{minipage}
\newline
\vspace{0.5cm}
\begin{minipage}{100mm}
		\textbf{Gianmarco Sperone}\\
		Facultad de Matemáticas\\
		Pontificia Universidad Católica de Chile\\
		Avenida Vicuña Mackenna 4860\\
		7820436 Santiago - Chile\\
		E-mail: gianmarco.sperone@uc.cl
\end{minipage}
}	
\end{document}